\documentclass[12pt]{article}
\usepackage{amssymb}
\usepackage{amsmath}
\usepackage{amsthm}
\usepackage{mathtools}
\usepackage[colorlinks=true,allcolors=black]{hyperref}
\usepackage{url}
\usepackage{setspace}
\usepackage{float}
\usepackage{tikz}
\tikzset{every picture/.style={line width=0.75pt}}
\usepackage[labelfont=bf]{caption}

\usepackage{mathtools}
\usepackage{enumerate}
\usepackage{mathrsfs}

\usepackage{setspace}
\usepackage{tikz}
\tikzset{every picture/.style={line width=0.75pt}}
\usepackage[labelfont=bf]{caption}

\newtheorem{theorem}{Theorem}
\newtheorem{proposition}{Proposition}
\newtheorem{definition}{Definition}
\newtheorem{remark}{Remark}
\newtheorem{lemma}{Lemma}

\newcommand{\mres}{\mathbin{\vrule height 1.6ex depth 0pt width
0.13ex\vrule height 0.13ex depth 0pt width 1.3ex}}

\NewDocumentCommand{\sff}{}{\mathrm{I\!I}}

\newcommand{\Addresses}{{
  \bigskip
  \footnotesize

KMS:
\textsc{Department of Mathematics, Johns Hopkins University, 404 Krieger Hall, 3400 North Charles Street, Baltimore, MD 21218, USA}\par\nopagebreak
  \textit{Email address:} \texttt{kmarsh34@jh.edu}

\medskip

GN:
\textsc{Department of Mathematics, University of Rochester, 915 Hylan Building, Rochester, NY 14620, USA}\par\nopagebreak
  \textit{Email address:} \texttt{gniu3@ur.rochester.edu}

}}
\begin{document}
\title{\large \textbf{ISOPERIMETRY OF SINGULAR CMC HYPERSURFACES}}
\author{\small KOBE MARSHALL-STEVENS \& GONGPING NIU}
\date{\vspace{-4ex}}
\maketitle
\begin{abstract}
    \noindent We produce the first examples of singular isoperimetric boundaries with non-zero constant mean curvature in Riemannian manifolds.~The main technical ingredient is establishing that, for constant mean curvature hypersurfaces containing isolated singularities modelled on strictly stable, strictly minimising regular tangent cones, weak strict stability implies local unique volume-constrained minimisation.
    \end{abstract} 
\tableofcontents

\section{Introduction}

Isoperimetric boundaries minimise area for a fixed enclosed volume, with regularity theory ensuring smoothness away from a closed singular set of codimension seven. In \cite{MN26} we showed that this result is sharp:~in a given closed manifold, one can construct a variety of singular isoperimetric boundaries saturating this dimension bound on the singular set. Since isoperimetric boundaries can a priori possess any constant mean curvature (CMC), one might expect that `most' isoperimetric boundaries should not be minimal. However, the constructions of \cite{MN26} rested on a transfer of local area-minimisation to global isoperimetry, and hence the resulting singular isoperimetric boundaries all arose as minimal hypersurfaces. Motivated by this, we establish here the following result:

\begin{theorem}\label{thm: CMC singular isop}
    On a closed manifold of dimension at least eight, there exist singular isoperimetric boundaries with non-zero constant mean curvature with respect to some Riemannian metric.
\end{theorem}

This shows the effectiveness of the results of \cite{MNP25}; there it was shown that for a generic choice of Riemannian metric and enclosed volume pair (or for fixed enclosed volume and generic Riemannian metric) on a closed eight manifold, every isoperimetric region with respect to this pair (or metric) has smooth boundary. While the generic regularity results of \cite{LW25} ensure that all minimal isoperimetric boundaries (e.g.~as produced in both \cite{MN26} and \cite{N24}) are smooth in a closed eight manifold equipped with a generic Riemannian metric, this does not account for isolated singularities arising in isoperimetric boundaries with non-zero CMC (e.g.~as produced by Theorem~\ref{thm: CMC singular isop}). Indeed, the isolated singularities in these hypersurfaces are only guaranteed to be perturbed away by the results of \cite{MNP25}, which account for any possible CMC value.

\bigskip

In \cite[Corollary 2 \& Remark 4]{MN26} we showed that any smooth, connected, null-cobordant manifold could arise as an isoperimetric boundary with non-zero CMC value. This rested on the transfer of strict stability to volume-constrained minimisation available for smooth hypersurfaces by the results of \cite{W94} and \cite{G96}; this is not known to hold generally in the presence of a singular set. Here, we establish this transfer for certain isolated singularities, generalising these results. This is the main technical result on which the constructions in Theorem~\ref{thm: CMC singular isop} are based:

\begin{theorem}\label{thm: strict stability to VCM}
    If $\Sigma = \partial \Omega$ is a strictly stable $H$-CMC hypersurface with isolated singularities modelled on strictly stable, strictly minimising regular tangent cones, then $\Omega$ uniquely minimises the natural area-type functional $\mathcal{F}_H$, and is consequently uniquely volume-constrained minimising, in a tubular neighbourhood of $\Sigma$. The conclusion of unique volume-constrained minimisation in a tubular neighbourhood also holds under the assumption that the CMC hypersurface is weakly strictly stable. 
\end{theorem}

For precise definitions of the natural area-type functional $\mathcal{F}_H$, strict stability of CMC hypersurfaces in the presence of isolated singularities, and the volume-constrained minimisation referred to above, see Subsection~\ref{subsec: notation}. As a consequence of Theorem~\ref{thm: strict stability to VCM}, the isoperimetric boundaries constructed in Theorem~\ref{thm: CMC singular isop} are unique for their enclosed volume; similar uniqueness holds for the examples produced in \cite{MN26} and \cite{N24} (see also the generic uniqueness results of \cite{N26}). Finally, we note that the minimal (i.e.~zero CMC value) case of Theorem~\ref{thm: strict stability to VCM} was established in \cite{S99}.

\subsection{Strategy}

In order to establish Theorem~\ref{thm: CMC singular isop} from Theorem~\ref{thm: strict stability to VCM} we utilise an adaptation of the `stretching' procedure introduced in \cite[Theorem 1]{MN26}. This reduces the problem to exhibiting connected, separating CMC hypersurfaces $\Sigma = \partial \Omega$ containing isolated singularities, such that $\Omega$ is volume-constrained minimising in a tubular neighbourhood of $\Sigma$.

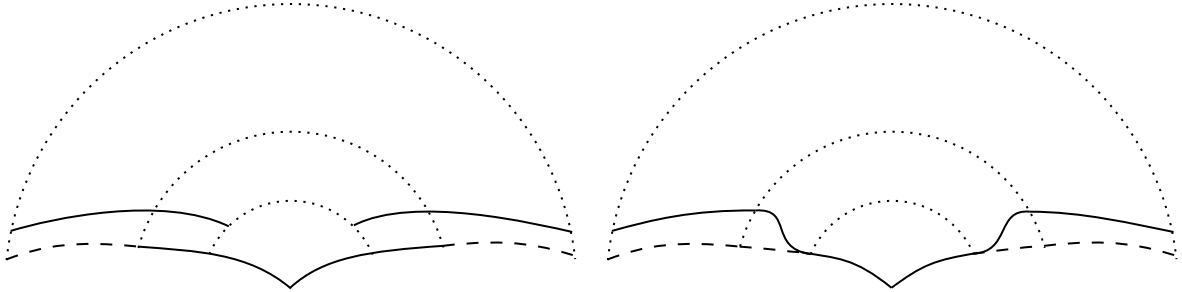
\begin{figure}[H]
    \centering
\captionsetup{justification=justified,margin=1cm}
\begin{tikzpicture}[scale=-0.02, xscale=-1]

\draw  [draw opacity=0][dash pattern={on 0.84pt off 2.51pt}] (279.68,207.31) .. controls (288.78,186.53) and (309.53,172.02) .. (333.67,172.02) .. controls (357.89,172.02) and (378.7,186.64) .. (387.76,207.54) -- (333.67,230.93) -- cycle ; \draw  [dash pattern={on 0.84pt off 2.51pt}] (279.68,207.31) .. controls (288.78,186.53) and (309.53,172.02) .. (333.67,172.02) .. controls (357.89,172.02) and (378.7,186.64) .. (387.76,207.54) ;  
\draw  [draw opacity=0][dash pattern={on 0.84pt off 2.51pt}] (232.76,203.2) .. controls (244.91,158.87) and (285.49,126.31) .. (333.67,126.31) .. controls (381.81,126.31) and (422.36,158.82) .. (434.55,203.09) -- (333.67,230.93) -- cycle ; \draw  [dash pattern={on 0.84pt off 2.51pt}] (232.76,203.2) .. controls (244.91,158.87) and (285.49,126.31) .. (333.67,126.31) .. controls (381.81,126.31) and (422.36,158.82) .. (434.55,203.09) ;  
\draw    (148.33,191.93) .. controls (175.67,183.6) and (249,167.85) .. (292.61,188.91) ;
\draw    (375.27,188.24) .. controls (415,167.85) and (495.67,187.27) .. (519.67,192.6) ;
\draw  [draw opacity=0][dash pattern={on 0.84pt off 2.51pt}] (145.61,211.27) .. controls (155.45,116.09) and (235.89,41.87) .. (333.67,41.87) .. controls (431.21,41.87) and (511.5,115.73) .. (521.65,210.58) -- (333.67,230.93) -- cycle ; \draw  [dash pattern={on 0.84pt off 2.51pt}] (145.61,211.27) .. controls (155.45,116.09) and (235.89,41.87) .. (333.67,41.87) .. controls (431.21,41.87) and (511.5,115.73) .. (521.65,210.58) ;  
\draw    (232.76,202.3) .. controls (271.26,205.3) and (305.75,206.2) .. (333.67,230.03) ;
\draw  [dash pattern={on 4.5pt off 4.5pt}]  (145.67,210.7) .. controls (171,201.2) and (191.5,198.2) .. (232.76,202.3) ;
\draw    (333.17,229.53) .. controls (355.25,209.2) and (385.25,205.2) .. (434.08,201.8) ;
\draw  [dash pattern={on 4.5pt off 4.5pt}]  (434.08,201.8) .. controls (479.75,196.2) and (501.75,202.2) .. (523.17,208.87) ;
\end{tikzpicture}
\hspace{1mm}
\begin{tikzpicture}[scale=-0.02, xscale=-1]

\draw  [draw opacity=0][dash pattern={on 0.84pt off 2.51pt}] (273.68,198.71) .. controls (282.78,177.93) and (303.53,163.42) .. (327.67,163.42) .. controls (351.07,163.42) and (371.29,177.07) .. (380.8,196.84) -- (327.67,222.33) -- cycle ; \draw  [dash pattern={on 0.84pt off 2.51pt}] (273.68,198.71) .. controls (282.78,177.93) and (303.53,163.42) .. (327.67,163.42) .. controls (351.07,163.42) and (371.29,177.07) .. (380.8,196.84) ;  
\draw  [draw opacity=0][dash pattern={on 0.84pt off 2.51pt}] (227.06,193.52) .. controls (239.57,149.75) and (279.88,117.71) .. (327.67,117.71) .. controls (375.81,117.71) and (416.36,150.22) .. (428.55,194.49) -- (327.67,222.33) -- cycle ; \draw  [dash pattern={on 0.84pt off 2.51pt}] (227.06,193.52) .. controls (239.57,149.75) and (279.88,117.71) .. (327.67,117.71) .. controls (375.81,117.71) and (416.36,150.22) .. (428.55,194.49) ;  
\draw  [draw opacity=0][dash pattern={on 0.84pt off 2.51pt}] (139.61,202.67) .. controls (149.45,107.49) and (229.89,33.27) .. (327.67,33.27) .. controls (425.21,33.27) and (505.5,107.13) .. (515.65,201.98) -- (327.67,222.33) -- cycle ; \draw  [dash pattern={on 0.84pt off 2.51pt}] (139.61,202.67) .. controls (149.45,107.49) and (229.89,33.27) .. (327.67,33.27) .. controls (425.21,33.27) and (505.5,107.13) .. (515.65,201.98) ;  
\draw  [dash pattern={on 4.5pt off 4.5pt}]  (226.76,193.7) .. controls (265.26,196.7) and (247.8,196.2) .. (268,197.5) ;
\draw  [dash pattern={on 4.5pt off 4.5pt}]  (139.67,202.1) .. controls (165,192.6) and (185.5,189.6) .. (226.76,193.7) ;
\draw  [dash pattern={on 4.5pt off 4.5pt}]  (428.08,193.2) .. controls (473.75,187.6) and (495.75,193.6) .. (517.17,200.27) ;
\draw    (268,197.5) .. controls (291.33,201.5) and (302.17,200.67) .. (327.17,220.93) ;
\draw    (327.17,220.93) .. controls (347.8,205.8) and (355,202.6) .. (384.6,197.8) ;
\draw  [dash pattern={on 4.5pt off 4.5pt}]  (380.88,198.4) .. controls (410.08,194.8) and (409,194.6) .. (428.08,193.2) ;
\draw    (237.4,169.8) .. controls (262.67,168.17) and (245.33,194.5) .. (273.68,198.71) ;
\draw    (142.33,183.33) .. controls (169.4,175.8) and (195.4,169.4) .. (237.4,169.8) ;
\draw    (418.6,170.6) .. controls (453.4,170.6) and (486.6,178.6) .. (513.67,184) ;
\draw    (380.88,198.4) .. controls (406.33,197.5) and (396,168.5) .. (418.6,170.6) ;
\end{tikzpicture}
    \caption{The graphic on the left above depicts smooth CMC graphs over a strictly stable minimal hypersurface (solid line outside the inner dotted ball) and a CMC hypersurface with an isolated singularity arising as a graph over a minimal hypercone (in the inner dotted ball). The solid line in the graphic on the right depicts a portion of $\Sigma$, obtained by interpolating (via gluing with a cutoff function) the two hypersurfaces (solid lines) depicted in the left graphic together in an annulus away from the singular point.}
    \label{fig: gluings}
\end{figure}

In order to construct $\Sigma$ as described above, we utilise an adaptation of the constructions of \cite{CHS84} (carried out in \cite{I17}) which produce CMC hypersurfaces with isolated singularities as graphs over regular minimal hypercones. The hypersurface $\Sigma$ is then constructed over strictly stable minimal hypersurfaces with isolated singularities as produced in \cite[Theorem 3]{MN26} (based on \cite{S99} and \cite{S00}): near each singular point $\Sigma$ is the CMC graph produced by \cite{CHS84} (where the metric is Euclidean), and away from the singular set $\Sigma$ is comprised of CMC graphs (which exist by strict stability of the minimal hypersurface). These are then glued to the \cite{CHS84} examples in annular regions around the isolated singularities. A conformal change in metric ensures the resulting hypersurface $\Sigma$ is then a CMC hypersurface. This construction is depicted in Figure~\ref{fig: gluings}. 

\bigskip

When the regular minimal hypercones in the construction of $\Sigma$ above are further assumed to be strictly stable, we show that $\Sigma$ is strictly stable (as defined in Subsection~\ref{subsec: notation}). By the above-mentioned reduction using the techniques of \cite{MN26}, to establish Theorem~\ref{thm: CMC singular isop} it remains only to show that $\Omega$ is volume-constrained minimising in a tubular neighbourhood of $\Sigma$, which is Theorem~\ref{thm: strict stability to VCM}.

\bigskip

We briefly recall how \cite{W94} establishes the transfer of strict stability to volume-constrained minimisation in the smooth case. Consider a smooth $H$-CMC hypersurface $\Sigma = \partial \Omega$ and the natural area-type functional $\mathcal{F}_H(E) = \mathrm{Per}(E) - H \cdot \mathrm{Vol}(E)$ defined on Caccioppoli sets $E \subset M$. One then constructs a sequence of smooth barrier hypersurfaces $\partial S_j$ (whose connected components converge to $\Sigma$ smoothly as graphs as $j \rightarrow \infty$) and observes that, since $\Sigma$ is smooth, for any sequence $\Sigma_j = \partial \Omega_j$ of $\mathcal{F}_H$-minimisers in $S_j$ with $\Omega_j \Delta \Omega \subset S_j$, we have that $\Sigma_j \rightarrow \Sigma$ as $C^{1,\alpha}$ graphs as $j \rightarrow \infty$. By computing a Taylor expansion of the functional $\mathcal{F}_H$ for graphs over $\Sigma$, the strict stability assumption ensures that $\mathcal{F}_H(\Omega) < \mathcal{F}_H(\Omega_j)$ for sufficiently large $j \geq 1$ unless $\Sigma_j = \Sigma$; one thus deduces that $\Sigma$ is uniquely $\mathcal{F}_H$-minimising (and hence volume-constrained minimising) in $S_j$ for $j \geq 1$ sufficiently large.

\bigskip

To emulate this approach when $\Sigma$ is a strictly stable $H$-CMC hypersurface containing isolated singularities modelled on strictly stable, strictly minimising regular tangent cones, we will require a control on $\mathcal{F}_H(\Omega_j)$ in balls around the singular points (where graphical convergence will fail). This control rests both on the construction of smooth barrier hypersurfaces $\partial S_j$ (as above) and on a height-to-graphing radius bound for a (now potentially singular) sequence $\Sigma_j$ which now only converges as $C^{1,\alpha}$ graphs to $\Sigma$ on compact subsets away from the singular points. The smooth barrier hypersurfaces $\partial S_j$ are constructed in a similar manner to those in \cite[Lemma 4]{S00} by exploiting the CMC analogue of \cite{HS85} provided in \cite{L23} and \cite{BL25}. We note that a similar height-to-graphing radius bound was deduced in the proof of \cite[Theorem 4]{MNP25} in order to ensure volume-control along sequences of convergent isoperimetric regions; a key difference here is that we deduce a super-quadratic (as opposed to super-linear) exponent under the stronger assumption of strict stability on the tangent cones. With this estimate in hand, we can excise the singular set and compute a Taylor expansion of $\mathcal{F}_H$ up to an error controlled by our height-to-graphing radius estimate. This error term can then be absorbed in the calculation, and the first statement of Theorem~\ref{thm: strict stability to VCM} then follows. 

\bigskip

A similar line of reasoning to the above, now accounting for an imposed volume constraint and emulating the smooth case in \cite{G96}, also applies when assuming that $\Sigma$ is only weakly strictly stable, establishing only the weaker conclusion of volume-constrained minimisation in this case; this gives the second statement of Theorem~\ref{thm: strict stability to VCM}.

\bigskip

We now proceed as follows: Section~\ref{sec: local behaviour} records both the adaptations of the constructions of \cite{CHS84} to the CMC setting and results on the growth rates of CMC hypersurfaces near isolated singularities, before establishing Theorem~\ref{thm: strict stability to VCM}. Section~\ref{sec: global behaviour} constructs various examples of singular CMC hypersurfaces over singular minimal hypersurfaces and concludes by establishing Theorem~\ref{thm: CMC singular isop}.

\subsection{Notation}\label{subsec: notation}

We now collect some notation and definitions that will be used throughout this work:

\begin{itemize}
    \item Let $(M,g)$ be a smooth, closed (i.e.~compact with empty boundary) Riemannian manifold of dimension $n+1 \geq 8$. Without loss of generality we will assume that $M$ is connected. Given a smooth Riemannian metric, $g$, on $M$ we denote $\mathrm{dist}_g$ for the distance function on $M$, $B^g_r(p)$ the open geodesic ball in $M$ of radius $r > 0$ centred at $p \in M$, $A^g(p;r,s)$ for the annulus of radii $0 \leq r < s$ centred at $p \in M$, $\exp_x^g$ for the exponential map based at $x \in M$, $\mathrm{Vol}_g(E)$ for the volume of a measurable set $E \subset M$, and $\int_E u \, dV_g$ for the integral of an integrable function $u$, all with respect to the metric $g$. In $\mathbb{R}^{n+1}$ we write $g_\mathrm{eucl}$ for the Euclidean metric and omit metric dependence from our notation.

    \item By a \textbf{hypersurface}, $\Sigma \subset M$, we will mean a closed subset $\Sigma = \mathrm{Reg}(\Sigma) \cup \mathrm{Sing}(\Sigma)$, where the \textbf{regular set}, $\mathrm{Reg}(\Sigma)$, is a smooth (or potentially $C^{1,\alpha} \cap W^{2,2}$ for some $\alpha \in (0,1)$) embedded manifold of dimension $n$, and the \textbf{singular set}, $\mathrm{Sing}(\Sigma)$, is a (possibly empty) closed set of Hausdorff dimension at most $n-7$. We write $B_r^g(\mathrm{Sing}(\Sigma))$ for the points $x \in M$ with $\mathrm{dist}_g(x,\mathrm{Sing}(\Sigma)) < r$ and $\mathrm{graph}_\Sigma^g(u)$ the normal graph of $u$ over $\Sigma$. Such hypersurfaces will arise as boundaries of Caccioppoli sets (defined below), which we may regard as integral $n$-currents with $\mathbb{Z}$ coefficients when the regular part is oriented, or $n$-chains with $\mathbb{Z}_2$ coefficients otherwise. We will also write $|\Sigma|_g$ for the multiplicity-one varifold associated to $\Sigma$ and denote the varifold metric
    \begin{equation*}
        \mathbf{F}(V, W) = \sup\{\Vert V \Vert(f) - \Vert W \Vert(f) \, | \, f \in C^1(M), \vert f \vert \leq 1, \vert Df \vert \leq 1 \}, 
    \end{equation*}
    defined for integer rectifiable varifolds $V$ where $\|V\|$ denotes the weight measure of $V$.

    \item A measurable set $E \subset M$ is a \textbf{Caccioppoli set} if the indicator function of $E$ is of bounded variation, or equivalently if it has finite \textbf{perimeter} with respect to the metric $g$ so that
     \begin{equation*}
        \mathrm{Per}_g(E) = \sup \left\{\int_E \mathrm{div}_g X \, dV_g \: \bigg| \: X \in \Gamma(TM), g(X,X) \leq 1 \right\} < \infty,
     \end{equation*}
     where $\mathrm{div}_g$ is the divergence with respect to the metric $g$ and $\Gamma(TM)$ is the set of vector fields on $M$; we write $\mathrm{Per}_g(E;U)$ for the perimeter relative to a measurable set $U \subset M$ by restricting the integral above to $U$. We write $\mathcal{C}(S)$ for the set of Caccioppoli sets in a set $S$, and note that $\mathcal{C}(S)$ is independent of the choice of metric on $S$. For a Caccioppoli set, $U \in \mathcal{C}(M)$, we denote by $[U]$ the associated integral $(n+1)$-current with coefficients in $\mathbb Z_2$; then $\partial U =\mathrm{Spt}(\partial[U])$.

     \item We say that $\Omega \in \mathcal{C}(M)$ is an \textbf{isoperimetric region} of enclosed volume $t \in (0,\mathrm{Vol}_g(M))$ if
     \begin{equation*}
        \mathrm{Per}_g(\Omega ) = \inf_{E \in \mathcal{C}(M)}\{\mathrm{Per}_g(E) \: | \: \mathrm{Vol}_g(E) = t\}.
     \end{equation*}
     The existence of isoperimetric regions of a given enclosed volume in $M$ is then guaranteed by the direct method of the calculus of variations (e.g.~see \cite[Section 13.2]{M16}). Moreover, we say that $\Omega \in \mathcal{C}(M)$ is a \textbf{volume-constrained minimiser} in an open set $U \subset M$ if 
     \begin{align*}
        \mathrm{Per}_g(\Omega; U ) = \inf_{E\in \mathcal{C}(M)} \{\mathrm{Per}_g(E; U) \, | \, \Omega \Delta E \subset \subset U \text{ and  } \mathrm{Vol}_g(\Omega\cap U) = \mathrm{Vol}_g(E\cap U) \}.
    \end{align*}
       
     \item Given an isoperimetric region $\Omega$ as above, we denote by $\Sigma =\partial \Omega$ the closed, two-sided, embedded CMC hypersurface; by \cite{GMT83} and \cite{M03}, we ensure that $\mathrm{Sing}(\Sigma)$ is of Hausdorff dimension at most $n-7$. We denote by $\nu_g$ the outward-pointing unit normal to $\Sigma$ with respect to the metric $g$.

     \item We often consider the area-type functional
     $$\mathcal{F}_H(E) = \mathrm{Per}_g(E) - H \cdot \mathrm{Vol}_g(E),$$
     defined on $E \in \mathcal{C}(M)$ with respect to $g$. Critical points of $\mathcal{F}_H$ have boundaries with distributional mean curvature equal to $H$, and minimisers of $\mathcal{F}_H$ are volume-constrained minimising in $M$.
 
     \item Suppose that $\mathbf{C} \subset \mathbb{R}^{n+1}$ is a \textbf{regular} minimal hypercone, so that $\mathrm{Sing}(\mathbf{C}) \subset \{0\}$. We then denote the smooth, closed hypersurface $S = \mathbf{C} \cap \mathbb{S}^n \subset \mathbb{R}^{n+1}$ as its \textbf{link}. By parameterising $\mathbf{C}$ in radial coordinates, $(r,\omega) \in (0,\infty) \times S$, we decompose (as in \cite{S68}) the Jacobi operator, $L_{\mathbf{C}}$, of $\mathbf{C}$ as
    $$L_{\mathbf{C}} = \partial^2_r + \left(\frac{n-1}{r}\right)\partial_r + \frac{1}{r^2}(\Delta_S + |\mathrm{I\!I}_S|^2),$$
    where $\mathrm{I\!I}_S$ is the second fundamental form of $S$ in $\mathbb{S}^n$. We let $\mu_1 < \mu_2 \leq \mu_3 \leq \ldots \nearrow + \infty$ be the eigenvalues of $- (\Delta_S + \vert \mathrm{I\!I}_S \vert^2)$, and $\varphi_1, \varphi_2, \ldots$ be the corresponding $L^2(S)$-orthonormal eigenfunctions where $\varphi_1 > 0$. By \cite{S68}$, \mathbf{C}$ is \textbf{stable} if and only if $\mu_1\geq -\frac{(n-2)^2}{4}$ and we say (as in \cite{CHS84}) that $\mathbf{C}$ is \textbf{strictly stable} if $\mu_1 > -\frac{(n-2)^2}{4}$; the relationships above regarding inequalities for $\mu_1$ and the stability of $\mathbf{C}$ are made explicit in \cite[Section 4]{CHS84}. Moreover, we say (as in \cite{HS85}) that a minimal hypercone, $\mathbf{C}$, as above is \textbf{strictly minimising} if there is a $\Theta > 0$ such that for each $\varepsilon > 0$, whenever $T$ is an integer rectifiable current with $\mathrm{Spt}(T) \subset \mathbb{R}^{n+1} \setminus B_\varepsilon(0)$ and $\partial T = \partial (\mathbf{C} \cap B_1(0))$ we have
    $$|\mathbf{C} \cap B_1(0)| \leq |T| - \Theta\varepsilon^n.$$
    For $\Lambda > 0$ we denote $$\mathscr{C}_\Lambda = \{\text{stable regular minimal hypercones, $\mathbf{C} \subset \mathbb{R}^{n+1}$, with $\Vert \mathbf{C} \Vert(B_1(0)) \leq \Lambda$}\}$$
    and write $\mathscr{C}$ for the stable regular minimal hypercones without density bound.
    
    \item By \cite{CHS84}, a general \textbf{Jacobi field}, $v \in C_\mathrm{loc}^\infty(\mathbf{C} \setminus \{0\})$, on $\mathbf{C}$ (i.e.~a solution to $L_\mathbf{C}v = 0$) is given by a linear combination of homogeneous Jacobi fields
    \begin{equation*}
        v(r, \omega) = \sum_{j \geq 1} ( v_j^+(r) + v_j^-(r)) \varphi_j(\omega), 
    \end{equation*}
    where for each $j \geq 1$ we write
    \begin{equation*}
        v_j^+(r) = c_j^+ \cdot r^{\gamma_j^+} \qquad \text{and} \qquad v_j^-(r) = \begin{cases}
    c_j^- \cdot r^{\gamma_j^-}, & \text{if $\mu_j > - \frac{(n - 2)^2}{4}$}, \\
    c_j^- \cdot r^{\gamma_j^-} \log(r) & \text{if $\mu_j = -  \frac{(n - 2)^2}{4}$}, 
        \end{cases}
    \end{equation*}
    for some constants $c_j^\pm \in \mathbb{R}$, and
    \begin{equation*}
        \gamma_j^\pm = \gamma_j^\pm(\mathbf{C}) = - \frac{n - 2}{2} \pm \sqrt{\mu_j + \frac{(n - 2)^2}{4}}.
    \end{equation*}
    We call the collection of exponents, $\gamma^\pm_j$ for $j \geq 1$, the \textbf{asymptotic spectrum} of $\mathbf{C}$, which we denote by $\Gamma(\mathbf{C})$. We observe that if we choose $\gamma\in(-\frac{n-2}{2},\gamma_1^+(\mathbf C))$, then whenever $\mathbf{C}$ is strictly stable we have that
    \begin{equation}\label{eqn: super-quadratic exponent}
        \frac{n}{1-\gamma} >2.
    \end{equation}

    \item Following \cite{S99} we define the weighted Sobolev space $W^{1,2,-2}_0(\Sigma)$ to be the completion of $C^1_c(\mathrm{Reg}(\Sigma))$ with respect to the norm
    $$
    \|u\|_{W^{1,2,-2}(\Sigma)}^2 = \int_\Sigma (|\nabla u|^2 + u^2 \rho^{-2}) \, dA_g,
    $$
    where $\rho : M \rightarrow [0,\infty)$ is the distance to $\mathrm{Sing}(\Sigma)$. We then say that $\Sigma$ is \textbf{strictly stable} if $\Sigma$ is a singular CMC hypersurface and
    there is $\lambda > 0$ such that
    \begin{equation}\label{eqn: strict stability inequality}
       \delta^2\mathrm{Area}_\Sigma(u) =  \int_\Sigma (|\nabla u|^2 - (|\sff_\Sigma|^2 + \mathrm{Ric}_g(\nu,\nu))u^2) \, dA_g 
        \geq \lambda \int_\Sigma u^2 \rho^{-2} \, dA_g
    \end{equation}
    for all $u \in W^{1,2,-2}_0(\Sigma)$; if \eqref{eqn: strict stability inequality} holds only for those $W^{1,2,-2}_0$ functions with zero integral, we say that $\Sigma$ is \textbf{weakly strictly stable}. We note that strict stability (and weak strict stability) with respect to the $\rho^{-2}$-weighted $L^2$ and $W^{1,2,-2}$ norms are equivalent up to a change in the choice of $\lambda$. Indeed, using the fact that $|\sff_\Sigma|^2 + \mathrm{Ric}_g(\nu,\nu) \leq C\rho^{-2}$ we have
    $$
    \int_\Sigma |\nabla u|^2 \, dA_g = \delta^2\mathrm{Area}_\Sigma(u) + \int_\Sigma (|\sff_\Sigma|^2 + \mathrm{Ric}_g(\nu,\nu))u^2 \, dA_g \leq \left(1+\frac{C}{\lambda}\right)       \delta^2\mathrm{Area}_\Sigma(u),
    $$
    and so
    $$
    \|u\|_{W^{1,2,-2}(\Sigma)}^2 = \int_\Sigma |\nabla u|^2 \, dA_g + \int_\Sigma u^2\rho^{-2} \, dA_g \leq \left(1+\frac{C+1}{\lambda}\right)       \delta^2\mathrm{Area}_\Sigma(u),
    $$
    so \eqref{eqn: strict stability inequality} can be phrased with the squared $W^{1,2,-2}$ norm of $u$ up to replacing $\lambda$ with $\frac{\lambda}{\lambda + C + 1}$.
    
    \item We let $\mathrm{inj}(x;M,g)$ denote the injectivity radius at $x \in M$ with respect to the metric $g$ and for each hypersurface $\Sigma = \partial\Omega$ we define the \textbf{regularity scale} $r_S(x)$ depending on $M,g,$ and $\Sigma$, at a point $x \in \mathrm{Reg}(\Sigma)$ to be the supremum among all $r \in (0, \mathrm{inj}(x;M,g)/2)$ such that both of the following properties hold:
        \begin{itemize}
            \item $r^{2} \Vert \mathrm{Rm}_g\Vert_{C^0(B_r^g(x))}+r^3\Vert \nabla \mathrm{Rm}_g\Vert_{C^0(B_r^g(x))} \le 1/10$, where $\mathrm{Rm}_g$ denotes the Riemann curvature tensor of $M$ with respect to the metric $g$.
            \item After pulling back by $\exp^g_x$ to $T_xM$ we have
                $$\frac1r (\exp^g_x)^{-1}(\Sigma) \cap B_1(0) = \mathrm{graph}^{g_\mathrm{eucl}}_L (u) \cap B_1(0),$$
          for some hyperplane $L \subset T_xM$ and $u \in C^3(L)$ with $\Vert u\Vert_{C^3(L)} \le 1/10$.
        \end{itemize}

    \item For $\phi\in C_{\mathrm{loc}}^k(\Sigma)$ with $k = 0,1,2$ we define the following norm on $E \subset \Sigma$ to be
        $$\Vert \phi\Vert_{C^k_*(E)}= \sup_{x \in E} \sum_{j=0}^k r_S(x)^{j-1}|\nabla^j \phi(x)|,$$
    which is invariant under the scaling of $\phi$ to $\lambda \phi$ and $g$ to $\lambda^2g$ for $\lambda > 0$.

    \item By \cite[Theorem 5]{S83B}, for any $p\in \mathrm{Sing}(\Sigma),$ we may express $\Sigma$ in \textbf{conical coordinates} near $p$ over its tangent cone $\mathbf{C}_p\Sigma$; precisely, for any $\varepsilon >0$, there exist $\phi\in C^2(\mathbf{C}_p\Sigma \setminus \{0\})$ and $r_p(\varepsilon)>0$ such that $\Vert \phi\Vert_{C_*^2((\mathbf{C}_p \cap B_{r_p(\varepsilon)}(0)) \setminus \{0\})}< \varepsilon$ and
    $$ \mathrm{graph}^{g_\mathrm{eucl}}_{\mathbf{C}_p\Sigma} (\phi)\cap 
    B_{r_p(\varepsilon)}(0) = (\exp_p^g)^{-1}\left(\Sigma\ \cap B^g_{r_p(\varepsilon)}(p)\right).$$ 
\end{itemize}

\section{Local behaviour}\label{sec: local behaviour}

In this section we recall the constructions of \cite{CHS84} (and the extension due to \cite{I17}) which produce prescribed mean curvature hypersurfaces with isolated singularities modelled on regular minimal hypercones. We then record the definition of growth rates for such hypersurfaces (c.f.~\cite{MNP25} and \cite{LW25}) and establish a growth rate monotonicity for local isoperimetric regions with respect to barriers. With this growth rate monotonicity in hand, we conclude the section by establishing both statements in Theorem~\ref{thm: strict stability to VCM}.

\subsection{Constructions over regular minimal hypercones}

We first recall the construction of \cite{CHS84} which produces singular minimal hypersurfaces modelled on regular minimal hypercones, explaining how modifications of their techniques produce prescribed mean curvature hypersurfaces; this extension was carried out explicitly in \cite{I17}.

\bigskip

In \cite{CHS84} examples of (embedded) minimal hypersurfaces in $\mathbb{R}^{n+1}$ with isolated singularities at the origin are constructed by perturbing the link, $\mathbf{C} \cap \mathbb{S
}^n$, of a given regular minimal hypercone $\mathbf{C}$ and solving $\mathcal{M}_\mathbf{C}u = 0$ where $\mathcal{M}_\mathbf{C}$ is the mean curvature operator of graphs of functions on $\mathbf{C}$. When the functions $u$ has appropriate decay towards the origin, the resulting graph 
$$\mathrm{graph}_{\mathbf{C}}(u) = \{x + u(x)\nu(x) \, | \, x \in \mathbf{C} \setminus \{0\}\},$$
where $\nu$ is an outward unit normal to $\mathbf{C}$, is embedded. Minimal hypersurfaces are constructed by first finding (by separation of the radial variable)
solutions of the inhomogeneous Jacobi field equation with appropriately fast decay to zero on approach to the origin, applying Schauder
estimates for small data in this linear problem, and utilising an appropriate fixed-point theorem to find the desired solutions of the non-linear problem. The same fixed-point arguments work for small enough inhomogeneous data $f$ and provide solutions to the prescribed mean curvature equation $\mathcal{M}_\mathbf{C}(u) = f$ (shown explicitly in \cite{I17}). More precisely, the results of \cite{CHS84} and \cite{I17} show:

\begin{theorem}[\cite{CHS84},\cite{I17}]\label{thm: CMC CHS} 
    For $n\ge 3$ and $\alpha \in (0,1)$, if $\mathbf{C} \subset \mathbb{R}^{n+1}$ is a regular minimal hypercone, then there exist constants $\delta \in [0,1) \setminus \Gamma(\mathbf{C})$ depending on $n$ and $\mathbf{C}$, and $\varepsilon \in (0,1), C > 0$ depending on $n, \alpha$, and $\mathbf{C}$ such that the following holds. For all 
    $f \in C^{0,\alpha}(\mathbf{C} \cap B_1(0) \setminus \{0\})$ with
    $$
    \sup_{0<r<\frac{1}{2}} \|f(r\, \cdot)\|_{C^{0,\alpha}(\mathbf{C} \cap A(0;1,2))}\le \varepsilon,
    $$
    there exists $u \in C^{2,\alpha}(\mathbf{C} \cap B_1(0) \setminus \{0\})$, with $\mathrm{graph}_{\mathbf{C}}(u)$ an embedded hypersurface, solving
    $$
    \mathcal{M}_{\mathbf{C}} u = f \quad \text{in } \mathbf{C} \cap B_1(0) \setminus \{0\},
    $$
    and satisfying the decay condition
    $$
    \sup_{0<r<\frac{1}{2}} r^{-2+\delta}\,\|u(r\, \cdot)\|_{C^{2,\alpha}(\mathbf{C} \cap A(0;\frac{1}{2},1))}< C \sup_{0 < r < \frac{1}{2}}\|f(r\, \cdot)\|_{C^{0,\alpha}(\mathbf{C} \cap A(0;\frac{1}{2},1))}.
    $$
    Moreover, when $f$ is constant and $\mathbf{C}$ is strictly stable, then, after potentially choosing $\varepsilon \in (0,1)$ smaller, $\mathrm{graph}_{\mathbf{C}}(u)$ is strictly stable in the sense of \eqref{eqn: strict stability inequality} in $B_1(0)$.
\end{theorem}

\subsection{Growth rates over regular minimal hypercones}

We now sharpen results in \cite[Appendix A]{MNP25} concerning the growth rates near isolated singular points of converging sequences of hypersurfaces that will be utilised in proving Theorem~\ref{thm: strict stability to VCM}. We first recall some definitions:

\begin{definition} \label{definition: asymptotic and growth rates} Given a hypersurface $\Sigma \subset M$, $p\in \mathrm{Sing}(\Sigma)$ isolated, $v\in L^2_{\mathrm{loc}}(\Sigma \setminus \{p\})$, $\mathbf{C} \in \mathscr{C}$, $w \in C^2_{\mathrm{loc}}(\mathbf{C} \setminus \{0\})$, $r > 0$, and $K > 1$, we define:

\begin{itemize}
    \item The \textbf{growth rate} of $v$ at $p$, Given $\gamma \in \mathbb{R}$, to be
        $$ J_{K;\Sigma,g}^{\gamma}(v;r) = \left(\int_{A(p;K^{-1}r, r)} v^2 \cdot \rho^{- n - 2\gamma} \, dA_g\right)^{\frac{1}{2}},$$
    where $\rho(x) = \mathrm{dist}_g(x, p)$. 

    \item The \textbf{asymptotic rate} of $w$ at infinity to be
        $$\mathcal{AR}_\infty(w) = \inf\left\{ \gamma \, \bigg| \, \lim_{s\nearrow +\infty} \int_{A(0;s,2s)} w^2(x) \cdot |x|^{-n-2\gamma} \, d\|\mathbf{C}\|(x) = 0 \right\},$$ 
    and the \textbf{growth rate} of $w$ on $\mathbf{C}$, given $\gamma \in \mathbb{R}$, to be
    \begin{equation*}
        J_{K;\mathbf{C}}^{\gamma}(w;r) = \left(\int_{A(0;K^{-1}r, r)} w^2(x) \cdot \vert x \vert^{- n - 2\gamma} \, d\|\mathbf{C}\|(x)\right)^\frac{1}{2}.
    \end{equation*}

    \item For a stationary integral varifold $V$ in $\mathbb{R}^{n+1} \setminus B_{R_0}$ for some $R_0 > 0$, we say that $V$ is \textbf{asymptotic to $\mathbf{C}$ at infinity} if there is $w \in C^2_\mathrm{loc}(\mathbf{C} \setminus \{0\})$ with
    $$V \mres (\mathbb{R}^{n+1} \setminus B_{R_1}) = |\mathrm{graph}^{g_\mathrm{eucl}}_{\mathbf{C} \setminus B_{R_1}}(w)|\mres (\mathbb{R}^{n+1} \setminus B_{R_1}) $$ for some $R_1 \geq R_0$ and $\Vert w\Vert_{C^2_*(A(0;R,2R))} \rightarrow 0$ as $R \rightarrow \infty$; for such a varifold $V$ asymptotic to $\mathbf{C}$ at infinity we write $\mathcal{AR}_\infty(V) = \mathcal{AR}_\infty(w)$. 
\end{itemize}
For the above definitions we adopt the convention that $\sup \emptyset = -\infty$ and $\sup \mathbb{R} = +\infty$.
\end{definition}

\begin{remark}\label{rem: J on cone to J on cmc} By expressing $\Sigma$ in conical coordinates over its tangent cone $\mathbf{C}_p\Sigma$, there is some constant $C > 0$, depending on the radius in which these coordinates are defined, such that
\begin{align*}
       C^{-1}J^{\gamma}_{K;\mathbf{C}}(u;r)^2 \leq J^{\gamma}_{K;\Sigma,g}(u;r)^2
\leq CJ^{\gamma}_{K;\mathbf{C}}(u;r)^2.
   \end{align*}
\end{remark}

We will utilise the following monotonicity property for the growth rate, whose statement and proof are an adaptation of the dichotomy appearing in \cite[Lemma F.5]{LW25}:

\begin{lemma}\label{lem: dichotomy adapted}
    For $\sigma > 0$, $\Lambda>\frac{n-2}{2}$, and $K > 2$ (chosen based on $\sigma,\Lambda$ as in \cite[Corollary 6.2]{LW25}), there exist constants $\delta, C > 0$, depending on $\sigma, \Lambda$, and $K$, for which the following holds. Consider a strictly minimising $\mathbf{C} \in \mathscr{C}_\Lambda$ and $\gamma \in (-\frac{n-2}{2}, \gamma_1^+(\mathbf{C}))$ such that
    $$
    \mathrm{dist}_{\mathbb{R}}\left(\gamma, \Gamma(\mathbf{C}) \cup \{-(n - 2)/2\}\right) \geq \sigma. 
    $$
    Let $g$ be a Riemannian metric on $B_2(0)$ so that in normal coordinates we have $\|g - g_{\mathrm{eucl}}\|_{C^4} \leq \delta$. Suppose that $\Omega \in \mathcal{C}(B_2(0))$, with $\Sigma = \partial \Omega$ an $H$-CMC hypersurface with respect to $g$ such that $\mathrm{Sing}(\Sigma) \cap B_2(0) = \{0\}$ and
    $$ |\Sigma| \mres B_2(0) = |\mathrm{graph}_\mathbf{C}(w)| \mres B_2(0) \text{ with } \|w\|_{C^2_*(\mathbf{C} \cap B_2(0))} < \delta.$$
    Suppose that $\{\Omega_j\}_{j \geq 1} \subset \mathcal{C}(B_2(0))$, with $\Sigma_j = \partial \Omega_j$ a $C^{1,\alpha} \cap W^{2,2}$ hypersurface with distributional mean curvature $H_j \in L^\infty(\Sigma_j)$ with respect to $g$ for each $j \geq 1$, such that $\|H_j - H\|_{L^\infty(\Sigma_j \cap B_2(0))} \rightarrow 0$ and $\mathbf{F}(|\Sigma|\mres B_2(0),|\Sigma_j| \mres B_2(0)) \rightarrow 0$ as $j \rightarrow \infty$. Then, if we define for each $j \geq 1$ the graphing radii $\tau_j$ and functions $v_j$ to be
    \begin{equation}\label{eqn: graphing radius adapted}
        \tau_j = \inf\{s \in (0,1) \mid |\Sigma_j| \mres A(0;s,1) = |\mathrm{graph}_{\Sigma}(v_j)| \mres A(0;s,1) \text{ and } \|v_j\|_{C^1_*(\Sigma \cap A(0;s,1))} \leq \delta \}
    \end{equation}
    we have for sufficiently large $j \geq 1$ and $t \in (2K^3\tau_j,1)$ that 
    \begin{equation}\label{eqn: J monotonicity with mean curvature difference}
        J^\gamma_{K;\Sigma,g}(v_j;K^{-1}t) \leq \max\{J^\gamma_{K;\Sigma,g}(v_j;t), C\|H_j - H\|_{L^\infty(\Sigma_j \cap B_2(0))}\}.
    \end{equation}
\end{lemma}

\begin{proof}
    We first note that in the case where $\Sigma_j$ has mean curvature $H$ (i.e.~$H_j = H$) for all $j \geq 1$, the proof reduces exactly to the statement of \cite[Lemma F.5]{LW25} since the equation for the difference of the graphs (see \eqref{eqn: weak form with mean curvature difference}) has right-hand side zero. The alternative in the dichotomy there is then ruled out by the strict minimisation assumption on $\mathbf{C}$ by \cite[Theorem 4.5(1)]{W20} (explained in full for the general case below).

    \bigskip

    The general case also follows by adapting the proof of \cite[Lemma F.5]{LW25} and appealing to \cite[Corollary 6.2]{LW25}. With the notation there we take $g_j^0 = g_j^1 = g$, $\Phi$ for $w$, and set $w_j = v_j \circ \Phi$ (the height of $\Sigma_j$ over $\Sigma$ pulled back to $\mathbf{C}$). Let $\varepsilon, K > 0$ be as in \cite[Corollary 6.2]{LW25}, and choose $\delta > 0$ sufficiently small so that, by \cite[Theorem B.1 and Remark 6.1]{LW25}, the error terms $\Vec{B}_0 ,B_1$ below satisfy \cite[(F.14)]{LW25} with $\varepsilon$ replaced by $\varepsilon/2$.
    Applying \cite[Theorem B.1]{LW25} we see that on $\mathbf{C} \cap A(0;2\tau_j,1)$ the functions $w_j$ weakly solve
    \begin{equation}\label{eqn: weak form with mean curvature difference}
        \mathrm{div}_\mathbf{C} (\nabla_\mathbf{C} w_j + \Vec{B}_0(x)) +|\mathrm{I\!I}_\mathbf{C}|^2 w_j + |x|^{-1} B_1(x) = \widetilde{H}_j,
    \end{equation}
    where $\Vec{B}_0 ,B_1$ satisfy the estimate \cite[(F.14)]{LW25} with $\frac{\varepsilon}{2}$. The right-hand side of \eqref{eqn: weak form with mean curvature difference} above is initially $(H_j\circ\Phi_j-H)F+H(F-F')$, where $F$ and $F'$ are the area densities of $\Sigma_j$ and $\Sigma$ respectively and $\Phi_j$ is the graph of $\Sigma_j$ over $\mathbf C$ as in \cite[Theorem B.1]{LW25}. However, since $|H|\le C_0\delta$ for some constant $C_0 > 0$ (since $\|g - g_{\mathrm{eucl}}\|_{C^4} \leq \delta$ and $\|w\|_{C^2_*(\mathbf{C} \cap B_2(0))} < \delta$), the second term is absorbed into $B_1$, and so we may set $\widetilde H_j =(H_j\circ\Phi_j-H)F$ with
    $$\|\widetilde{H}_j\|_{L^\infty(\mathbf{C} \cap A(0;2\tau_j,1))} \leq C_1\|H_j - H\|_{L^\infty(\Sigma_j \cap B_2(0))}$$ for a constant $C_1 > 0$ (into which we absorb the constant in Remark~\ref{rem: J on cone to J on cmc}). We now set $C_2=\frac{\varepsilon}{2}K^{-n}$ and $C=C_1C_2^{-1}$; this latter $C$ is the constant we will obtain for the statement.
    
    \bigskip

    For each $t \in (2K^3\tau_j,1)$ the rescalings $\hat{w}_j(x) = \frac{w_j(tx)}{t}$ weakly solve \eqref{eqn: weak form with mean curvature difference} on $\mathbf{C} \cap A(0;K^{-3},1)$ where now $\|\widetilde{H}_j\|_{L^\infty(\mathbf{C} \cap A(0;K^{-3},1))} \leq t C_1 \|H_j - H\|_{L^\infty(\Sigma_j \cap B_2(0))}$. Thus, if \eqref{eqn: J monotonicity with mean curvature difference} fails at this $t$ then we have
    \begin{align*}
        \|\widetilde H_j\|_{L^\infty(\mathbf C\cap A(0;K^{-3},1))}&\leq tC_1\|H_j - H\|_{L^\infty(\Sigma_j \cap B_2(0))}
        \leq tC_2 J^\gamma_{K;\mathbf C}( w_j;K^{-1}t )\\
        &\le C_2 t^{\gamma-1} J^\gamma_{K;\mathbf C}(w_j;K^{-1}t) \leq C_2 J^\gamma_{K;\mathbf C}(\hat w_j;K^{-1})\leq \frac{\varepsilon}{2}\|\hat w_j\|_{L^2(A(0;K^{-3},1))},
    \end{align*} 
    using Remark~\ref{rem: J on cone to J on cmc}, $t\le1$, $\gamma<0$, and $|x|^{-n-2\gamma}\le K^{2n}$ on $A(0;K^{-2},K^{-1})$. Absorbing $\widetilde H_j$ into $B_1$ we see that $\hat w_j$ satisfies \cite[(33)–(34)]{LW25} and so iterating \cite[Corollary 6.2]{LW25} gives \cite[(F.15)]{LW25} with the $\kappa\|g^1-g^0\|_{C^4}$ term appearing there replaced by $C\|H_j - H\|_{L^\infty(\Sigma_j \cap B_2(0))}$.

    \bigskip
    
    Then, if \eqref{eqn: J monotonicity with mean curvature difference} fails for infinitely many $j \geq 1$ the same reasoning following \cite[(F.15)]{LW25} shows that $\tau_j > 0$ for sufficiently large $j \geq 1$ and hence $\tau_j^{-1}|\Sigma_j|$ converges to a stationary integral varifold $V_\infty$ which differs from $\mathbf{C}$ but is asymptotic to it at infinity with $\mathcal{AR}_\infty(V_\infty) < \gamma$. However, this conclusion is precluded by the choice of $\gamma \in (-\frac{n-2}{2}, \gamma_1^+(\mathbf{C}))$ since the assumption that $\mathbf{C}$ is strictly minimising ensures $\mathcal{AR}_\infty(V_\infty) \geq \gamma_1^+$ by \cite[
    Theorem 4.5(1)]{W20}. Thus the desired statement holds. 
\end{proof}

\begin{remark}
    We note that the above lemma differs from the dichotomy stated in \cite[Appendix A]{MNP25} in the choice of $\gamma$, the fact that we do not consider changing metrics, and that we only work with the $C^1_*$ norm. This is permitted since the weak formulations used in the proof require only $W^{1,2}_\mathrm{loc}$ functions. 
\end{remark}

\subsection{Strict stability to volume-constrained minimisation}\label{subsec: strict stability to VCM}

Throughout this subsection we will assume $\Sigma = \partial \Omega$ is a connected, separating $H$-CMC hypersurface with finitely many isolated singularities modelled on strictly stable, strictly minimising regular tangent cones of multiplicity one; the results of \cite{S83B} show that such tangent cones are unique. Under the assumption of strict stability we will show that $\Omega$ is $\mathcal{F}_H$-minimising in a tubular neighbourhood of $\Sigma$. Moreover, under the assumption of weak strict stability we show $\Omega$ is volume-constrained minimising in a tubular neighbourhood of $\Sigma$. Combining both statements yields Theorem~\ref{thm: strict stability to VCM}.

\bigskip

We first construct a sequence of barriers for $\Sigma$, which do not require any global stability assumptions. More precisely, for each $j \geq 1$ we construct nested sets $A_j \subset \Omega \subset B_j$ converging to $\Omega$ in the $L^1$ sense as $j \rightarrow \infty$. Moreover, denoting $S_j = B_j \setminus A_j$, we show that each $\partial S_j$ (disjoint from $\Sigma$) is a pair of smooth hypersurfaces converging smoothly as graphs to $\Sigma$ away from each singular point as $j \rightarrow \infty$. 

\bigskip

As shown in \cite[Theorem 7.3]{L23} (c.f.~\cite[Theorem 4.1]{HS85}), the assumption of strict stability and strict minimisation of the regular tangent cone at each singular point ensures that $\Sigma$ is in fact locally $\mathcal{F}_H$-minimising around each singular point. Fix $r>0$ sufficiently small so that the balls $B_{8r}^g(p)$ are pairwise disjoint for each $p \in \mathrm{Sing}(\Sigma)$,  so that normal coordinates over $\Sigma\setminus\{\rho<r\}$ are defined, and sufficiently small for the application of \cite[Theorem 8.7]{L23} that follows below.

\bigskip

By \cite[Theorem 8.7]{L23} (which records the adaptation of \cite[Theorem 2]{BL25} to the Riemannian setting), for a sequence $t_j \rightarrow 0$ as $j \rightarrow \infty$ there are $\mathcal{F}_H$-minimisers $R_p^{\pm t_j}$ in $B_{8r}^g(p)$ with boundaries  equal to the $\pm t_j$-normal graphs over $\Sigma\cap\partial B_{8r}^g(p)$. We set $S_p^{\pm t_j} = \partial(R_p^{\pm t_j} )\mres B_{8r}^g(p)$, so that each $S_p^{\pm t_j}$ is an entirely smooth $H$-CMC hypersurface lying on one side of $\Sigma$ (we take the convention that $R^{-t_j}_p \subset \Omega \cap B_{8r}^g(p) \subset R^{+t_j}_p$). Moreover, we have that $S_p^{\pm t_j}\rightarrow \Sigma\cap B_{8r}^g(p)$ in Hausdorff distance as $t_j \rightarrow 0$, and thus on the annulus $r\leq \rho \leq 7r$ we can write each hypersurface as a smooth normal graph $w_p^{\pm t_j}$ over $\Sigma$ by Allard's theorem (e.g.~as stated in \cite[Chapter 5]{S83A}) whenever $t_j$ is sufficiently small. By the Harnack inequality (which we can apply since both $\Sigma$ and $S_p^{\pm t_j}$ are $H$-CMC) it then follows that for some $c > 0$ we have both
    $$
          c^{-1}|t_j|\le |w_p^{\pm t_j}|\le c|t_j| \quad \text{ and } \quad
            \max_{2r \leq \rho \leq 6r}|\nabla^k w_p^{\pm t_j}|
            \le c|t_j| \quad \text{ for }
            \quad k=1,2 .
    $$
    
Also by the choice of $r > 0$ above, for $|t|$ small we can define the smooth parallel hypersurfaces
    $$
            C_t
            =
            \{\exp_x(t\nu(x)) \, | \, x\in \Sigma,\ \rho(x)\ge r\},
    $$
where $\nu(x)$ is the outward-pointing unit normal to $\Sigma$ at a point $x \in \mathrm{Reg}(\Sigma)$. Since $\Sigma$ is an $H$-CMC hypersurface we compute directly that the mean curvature $H_{C_t}$ of the smooth hypersurfaces $C_t$ satisfies
    $$
            |H_{C_t}-H|\le C|t|
            \quad\text{on }\{\rho\ge r\},
    $$
for some constant $C > 0$ depending on $r$ and $\Sigma$.

\bigskip
    
We now glue the above hypersurfaces to form barriers. For this, we choose a smooth increasing cutoff function $\eta: [0,\infty) \rightarrow \mathbb{R}$, with $\eta = 0$
 on $[0,3r]$ and $\eta = 1$ on $[5r,\infty)$, and on each annulus $\{2r \leq \rho \leq 6r\}$ around $p \in \mathrm{Sing}(\Sigma)$ we set
    $$
            u_p^{\pm t_j}
            =
            (1-\eta)w_p^{\pm t_j} \pm \eta t_j;
    $$
    thus the graph of $u_p^{\pm t_j}$ interpolates between $C_{\pm t_j}$ and $S_p^{\pm t_j}$. We denote the resulting hypersurfaces by $\Gamma_j^{\pm}$ so that $S_j$ is defined implicitly to be the region with $\partial S_j = \Gamma_j^+ \cup \Gamma_j^-$ containing $\Sigma$. More concretely, we define
    $$
    \Gamma_j^{\pm} = \bigcup_{p \in \mathrm{Sing}(\Sigma)}(S_p^{\pm t_j} \cap B_{3r}^g(p)) \cup \mathrm{graph}(u_p^{\pm t_j}) \cup (C_{\pm t_j} \cap (M \setminus B_{5r}^g(p)))
    $$
    which are then smooth hypersurfaces. 
    
    \bigskip
    
    Finally, for each $j \geq 1$ we write $M = A_j \cup S_j \cup B_j^c$ so that $S_j = B_j \setminus A_j$ is chosen with outward-pointing unit normal to $\Gamma_j^+$ pointing into $B_j^c$; by construction we see that then for each $j \geq 1$ we have $A_j \subsetneq \Omega \subsetneq B_j$ with $|A_j \Delta \Omega|, |B_j \Delta \Omega|, |A_j \Delta B_j| \rightarrow 0$ as $j \rightarrow \infty$. Note that $\partial S_j$ then has two connected components $\Gamma^\pm_j$ with $\partial A_j = \Gamma_j^-$, $\partial B_j = \Gamma_j^+$, and $\partial S_j = \partial A_j \cup \partial B_j$ for each $j \geq 1$.

    \bigskip

    We will also make use of the following Taylor expansions in the proof of Theorem~\ref{thm: strict stability to VCM}:
    
\begin{remark}\label{rem: Taylor expansion}(Taylor expansions)
    Let $\Sigma=\partial \Omega$ be an $H$-CMC hypersurface with finitely many isolated singularities with strictly stable, strictly minimising regular tangent cones of multiplicity one as above. For $\psi\in C^1_*(\Sigma)$ denote by $\Omega_{\psi}$ the region obtained from $\Omega$ with boundary given by
    $$
      \Sigma_{\psi}= \left\{
        \exp_x^g\left(\psi(x)\nu(x)\right)
        \, | \,
        x\in\Sigma
        \right\}.
    $$
    We consider the Taylor expansions of the perimeter and volume. Following the computations in \cite[Appendix B]{LW25}, we see that there exist constants $\varepsilon_0,C>0$, depending on $\Sigma$ and $(M,g)$ such that for $\|\psi\|_{C^1_*(\Sigma)}\le \varepsilon_0$ we have
    $$
    \mathrm{Per}_g(\Omega_{\psi})
    -\mathrm{Per}_g(\Omega)
    =
    H\int_\Sigma\psi\,dA_g + 
    \frac{1}{2}\int_\Sigma
    \left(
    |\nabla_\Sigma\psi|^2-(|\sff_\Sigma|^2+\mathrm{Ric}_g(\nu,\nu)-H^2) \psi^2
    \right)\,dA_g +R_P(\psi),
    $$
    and 
    $$
    \mathrm{Vol}_g(\Omega_{\psi}) - \mathrm{Vol}_g(\Omega)
    =
    \int_\Sigma\psi\,dA_g
    +\frac{H}{2}\int_\Sigma\psi^2\,dA_g
    +R_V(\psi).
    $$
    
    where
    $$
        |R_P(\psi)|+|R_V(\psi)|
    \leq
    C\|\psi\|_{C^1_*(\Sigma)}\cdot
    \|\psi\|_{W^{1,2,-2}(\Sigma)}^2.
    $$
    Hence, writing $Q_L(\psi,\psi) = \int_\Sigma
    \left(
    |\nabla_\Sigma\psi|^2-(|\sff_\Sigma|^2+\mathrm{Ric}_g(\nu,\nu))\psi^2
    \right)\,dA_g$, we have

    \begin{equation}\label{eqn: Taylor expansion of F_H}
        \mathcal F_H(\Omega_{\psi})-\mathcal F_H(\Omega)
    =
    \frac{1}{2}Q_L(\psi,\psi)
    +R_F(\psi),
    \end{equation}
    where
    \begin{equation}\label{eqn: Taylor expansion error}
      | R_F(\psi)| = |R_P(\psi)-H R_V(\psi)| \leq
    C\|\psi\|_{C^1_*(\Sigma)}\cdot
    \|\psi\|_{W^{1,2,-2}(\Sigma)}^2.
    \end{equation}
    \end{remark}

\medskip

We now establish the first statement in Theorem~\ref{thm: strict stability to VCM}, which follows from:

\begin{theorem}\label{thm: F_H strict stability to VCM} If $\Sigma$ is strictly stable, then $\Omega$ is uniquely locally $\mathcal{F}_H$-minimising in a tubular neighbourhood of $\Sigma$. In particular, $\Omega$ is volume-constrained minimising in this tubular neighbourhood.
\end{theorem} 

\begin{proof} 
    We will show that $\Omega$ is uniquely $\mathcal{F}_H$-minimising in the sets $S_j$ constructed above for sufficiently large $j \geq 1$; both statements then follow since $\mathcal{F}_H$-minimisers are volume-constrained minimisers (seen by considering fixed-volume competitors). To this end, for each $j \geq 1$ we let $\Omega_j \in \mathcal{C}(M)$ solve
    $$
        \mathcal{F}_H(\Omega_j) = \inf_{E \in \mathcal{C}(M)}\{\mathcal{F}_H(E) \, | \, E \Delta \Omega \subset S_j\};
    $$
    we will show that $\Omega_j = \Omega$ for $j \geq 1$ sufficiently large. Note first that by regularity results for the isoperimetric problem with respect to an obstacle (e.g.~see \cite[Theorem 1.1]{V10} for a formulation in Riemannian manifolds) and since $\Omega_j$ is volume-constrained minimising (since it is $\mathcal{F}_H$-minimising) in $S_j$ and $\partial S_j$ is a union of smooth hypersurfaces, we have for some $\alpha \in (0,1)$ that $\Sigma_j = \partial \Omega_j$ is a $C^{1,\alpha}\cap W^{2,2}$ hypersurface away from a closed set of Hausdorff dimension at most $n - 7$ with bounded distributional mean curvature. Moreover, since $\Gamma_j^{\pm}$ are smooth $H$-CMC hypersurfaces in the balls $B_{2r}^g(p)$ for each $p \in \mathrm{Sing}(\Sigma)$, the mean curvature estimates in \cite[Theorem 1.1]{V10} ensure that $\Sigma_j$ is a (potentially singular) $H$-CMC hypersurface in each $B_{2r}^g(p)$ for $j \geq 1$ sufficiently large.

    \bigskip

    Since $\Omega_j \Delta \Omega \subset S_j$ for each $j \geq 1$, by their construction we have $A_j \subset \Omega_j \subset B_j$ and thus as both $|A_j \Delta \Omega|, |B_j \Delta \Omega| \rightarrow 0$ we have $|\Omega_j \Delta \Omega| \rightarrow 0$ as $j \rightarrow \infty$ also. Noting that the $\mathcal{F}_H$-minimisation of $\Omega_j$ gives
    $$
            \mathrm{Per}(\Omega_j)-H|\Omega_j|
            \le
            \mathrm{Per}(\Omega)-H|\Omega|,
    $$
    the fact that $|\Omega_j \Delta \Omega| \rightarrow 0$ as $j \rightarrow \infty$ combined with the lower semicontinuity of perimeter gives $ \mathrm{Per}(\Omega_j)\rightarrow\mathrm{Per}(\Omega)$ as $j \rightarrow \infty$. Thus, by \cite[Proposition A.1]{DT13}, we see that $\mathbf{F}(|\Sigma|,|\Sigma_j|) \rightarrow 0$ as $j \rightarrow \infty$ and hence $\Sigma_j \rightarrow \Sigma$ as $C^{1,\alpha}$ graphs away from singular points of $\Sigma$ by Allard's theorem (e.g.~as stated in \cite[Chapter 5]{S83A}) as $j \rightarrow \infty$.

    \bigskip

    For each $p \in \mathrm{Sing}(\Sigma)$ fix some $\gamma_p\in(-\frac{n-2}{2},\gamma_1^+(\mathbf C_p))$ a distance $\sigma > 0$ away from the asymptotic spectrum, and choose $\gamma \in (-\frac{n-2}{2},\min_{p \in \mathrm{Sing}(\Sigma)}(\gamma_p))$. For this choice of $\sigma > 0$ and some $\Lambda > \frac{n-2}{2}$ we then apply Lemma~\ref{lem: dichotomy adapted} to obtain $\delta>0$ for which its conclusions hold. Using conical coordinates for $\Sigma$ around each $p \in \mathrm{Sing}(\Sigma)$ (as defined in Subsection~\ref{subsec: notation}) we can write the graphing radii for $\Sigma_j$ over $\Sigma$ around each $p \in \mathrm{Sing}(\Sigma)$ as
    $$
         s_j = \max_{p \in \mathrm{Sing}(\Sigma)}\inf \{ \tau>0 \, | \, \|v_j\|_{C^1_*(A^g(p;\tau,1))}\le \delta\},
     $$
    this ensures for $j \geq 1$ sufficiently large that
        $$\Sigma_j \setminus B_{s_j}^g(p) = \mathrm{graph}_{\Sigma}^g(v_j) \setminus B_{s_j}^g(p).$$ 
    Moreover, by the strict stability and strict minimisation assumption on $\mathbf{C}_p$ for each $p \in \mathrm{Sing}(\Sigma)$, we have that both (\ref{eqn: super-quadratic exponent}) holds for $\gamma$ and (\ref{eqn: J monotonicity with mean curvature difference}) holds over each $\mathbf{C}_p$. 

    \bigskip
    
    By defining $a_j = \Vert v_j\Vert_{L^2(\Sigma \setminus B_{K^{-2}}^g(\mathrm{Sing}(\Sigma)))}$ we will now show, as in the proof of \cite[Theorem 4]{MNP25}, that 
    \begin{equation}\label{eqn: graphing radius estimate}
    s_j \leq a_j^{\frac{1}{1-\gamma}} = a_j^\frac{2 + \beta}{n}
    \end{equation}
    where $\gamma$ is chosen as above and $\beta > 0$ (the second equality coming from (\ref{eqn: super-quadratic exponent})). Indeed, if this fails then by defining $\hat{v}_j(x) = \frac{v_j(s_j x)}{s_j}$ on $A(0;1,\frac{1}{Ks_j})$ we see that $J^{\gamma_p}_{K;\mathbf{C}_p}\left(\hat{v}_j;\frac{1}{Ks_j}\right) \rightarrow 0$ since, by writing $y = s_jx$, we compute
    $$J^{\gamma_p}_{K;\mathbf{C}_p}\left(\hat{v}_j;\frac{1}{Ks_j}\right)^2 = s_j^{-2+2\gamma_p} \int_{A\left(0;\frac{1}{K^2},\frac{1}{K}\right)} v_j(y)^2|y|^{-n-2\gamma_p} \, d\|\mathbf{C}_p\|(y) \leq K^{2n + 4\gamma_p}s_j^{-2+2\gamma_p} a_j^2,$$
    
    so that under the assumption that $a_j \leq s_j^{1-\gamma}$ we ensure that
    $$J^{\gamma_p}_{K;\mathbf{C}_p}\left(\hat{v}_j;\frac{1}{Ks_j}\right) \leq Cs_j^{(\gamma_p - \gamma)} \rightarrow 0.$$
    
    However, by writing (\ref{eqn: J monotonicity with mean curvature difference}) for $\hat{v}_j$ (with $H_j = H$ for all $j \geq 1$ since the $\Sigma_j$ are $H$-CMC in $B_{2r}^g(p)$ for each $p \in \mathrm{Sing}(\Sigma)$) we deduce that $J^{\gamma_p}_{K;\mathbf{C}_p}(\hat{v}_j;K^l)$ is increasing for integers $3 \leq l \leq \log_K(\frac{1}{s_j}) - 1$ and hence we have
    $$J^{\gamma_p}_{K;\mathbf{C}_p}\left(\hat{v}_j;K^3\right) \leq J^{\gamma_p}_{K;\mathbf{C}_p}\left(\hat{v}_j;\frac{1}{Ks_j}\right).$$
    Since $\hat{v}_j \rightarrow h \in C^2_{\mathrm{loc}}(\mathbf{C}_p \setminus B_1(0))$ in $C^2_\mathrm{loc}$ where $h$ is the graph of a stationary varifold, $V$, which is asymptotic to but not equal to $\mathbf{C}_p$ (obtained as a limit of $s_j^{-1}|\Sigma_j|$ by the minimality of $s_j$),
    for all large $j$ we see that
    $$J^{\gamma_p}_{K;\mathbf{C}_p}\left(\hat{v}_j;K^3\right) > \frac{1}{2}J^{\gamma_p}_{K;\mathbf{C}_p}\left(h;K^3\right)  > 0$$
    
    where the right-hand term is strictly positive since $h \neq 0$ (since $V \neq \mathbf{C}_p$). This is a contradiction since we saw that $J^{\gamma_p}_{K;\mathbf{C}_p}\left(\hat{v}_j;\frac{1}{Ks_j}\right) \rightarrow 0$ as $j \rightarrow \infty$; hence (\ref{eqn: graphing radius estimate}) holds as desired.

    \bigskip

    Next, we modify (in a similar manner to Figure~\ref{fig: gluings}) the sets $\Omega_j$ inside of $B_{2s_j}^g(\mathrm{Sing}(\Sigma))$ in order to construct sets $\widetilde\Omega_j$ agreeing with $\Omega$ in $B_{s_j}^g(\mathrm{Sing}(\Sigma))$. To this end, we choose a smooth increasing cutoff function $\zeta\in C^\infty([0,\infty))$ with $\zeta = 0$ on $[0,1]$ and $\zeta = 1$ on $[2,\infty)$. For $j \geq 1$ and $x\in B_{2s_j}^g(p)\cap \Sigma$ for each $p \in \mathrm{Sing}(\Sigma)$ we define
    \begin{equation}\label{eqn: cutoff graphs}
        \widetilde v_j(x)
        = \zeta\!\left(\frac{\mathrm{dist}_g(x,p)}{s_j}\right)v_j(x)
    \end{equation}
    for $x \in A^g(p;s_j,2s_j)$ while setting $\widetilde v_j = v_j$ on $\Sigma \setminus B_{2s_j}^g(p)$ and $\widetilde v_j = 0$ on $B_{s_j}^g(p)$ for each $p \in \mathrm{Sing}(\Sigma)$. Since for each $j \geq 1$ we have $\|v_j\|_{C^1_*(A^g(p;s_j,1))} \leq \delta$ for each $p \in \mathrm{Sing}(\Sigma)$, we ensure
    $\|\widetilde v_j\|_{C^1_*(A^g(p;s_j,1))} \leq C \delta$ for a constant $C > 0$ depending on $\zeta$. Moreover, we note that $a_j \leq \|\widetilde{v}_j\|_{L^2(\Sigma)}$ for each $j \geq 1$ by construction.
    
    \bigskip
    
    For each $j \geq 1$ we then denote by $\widetilde\Omega_j$ the Caccioppoli set whose
    boundary is given by the global normal graph $\widetilde\Sigma_j = \mathrm{graph}^{g}_{\Sigma}(\widetilde v_j)$, with $\widetilde{\Omega}_j \Delta \Omega_j \subset S_j$. The calculations in Remark~\ref{rem: Taylor expansion} then provide an analogous computation to that of \cite[(1) in the proof of Theorem 2]{W94} and show that
    
    $$\mathcal{F}_H(\widetilde\Omega_j) - \mathcal{F}_H(\Omega) = \frac{1}{2}Q_L(\widetilde{v}_j,\widetilde{v}_j) + R_F(\widetilde{v}_j) \geq \frac{\lambda}{2}\|\widetilde{v}_j\|_{W^{1,2,-2}(\Sigma)}^2 - C\|\widetilde{v}_j\|_{C^1_*(\Sigma)}\|\widetilde{v}_j\|_{W^{1,2,-2}(\Sigma)}^2$$
    where we used \eqref{eqn: strict stability inequality} from the strict stability (with constant $\lambda > 0$) of $\Sigma$ with respect to the $W^{1,2,-2}$ norm (as in Subsection~\ref{subsec: notation}) and \eqref{eqn: Taylor expansion error} to bound $R_F$.

    \bigskip
    
    Noting that $\widetilde\Omega_j$ differs from $\Omega_j$ only in $B_{2s_j}^g(\mathrm{Sing}(\Sigma))$, we see that
    $$|\mathcal{F}_H(\widetilde\Omega_j) - \mathcal{F}_H(\Omega_j)| \leq Cs_j^n,$$

    for a constant $C > 0$ depending on $\Sigma$ and $(M,g)$, and thus for $j \geq 1$ sufficiently large we have 

    $$\mathcal{F}_H(\Omega_j) - \mathcal{F}_H(\Omega) \geq  \frac{\lambda}{2}\|\widetilde{v}_j\|_{W^{1,2,-2}(\Sigma)}^2  - \|\widetilde{v}_j\|_{C^1_*(\Sigma)}\|\widetilde{v}_j\|_{W^{1,2,-2}(\Sigma)}^2 - Cs_j^n.$$

    Combining $a_j \leq \|\widetilde{v}_j\|_{L^2(\Sigma)}$ with \eqref{eqn: super-quadratic exponent} and \eqref{eqn: graphing radius estimate} we see that $s_j^n \leq \|\widetilde{v}_j\|_{L^2(\Sigma)}^{2 + \beta}$ for some $\beta > 0$. Moreover, if we choose $\delta < \frac{\lambda}{8C}$ above so that $\|\widetilde{v}_j\|_{C^1_*(\Sigma)} < \frac{\lambda}{8}$ (noting that Lemma~\ref{lem: dichotomy adapted} holds for all sufficiently small $\delta > 0$) and use the fact that $\|\widetilde{v}_j\|_{L^2(\Sigma)} \rightarrow 0$ we see that for $j \geq 1$ sufficiently large we have
    
    $$\mathcal{F}_H(\Omega_j) - \mathcal{F}_H(\Omega) \geq  \frac{\lambda}{4}\|\widetilde{v}_j\|_{W^{1,2,-2}(\Sigma)}^2.$$

    The above inequality ensures that, if $\Omega_j \neq \Omega$ we must have
    
    $$\mathcal{F}_H(\Omega_j) > \mathcal{F}_H(\Omega)$$

    whenever $j \geq 1$ is sufficiently large. This contradicts the assumption that the $\Omega_j$ were $\mathcal{F}_H$-minimising in $S_j$ unless $\Omega_j = \Omega$. Thus, $\Omega$ is uniquely $\mathcal{F}_H$-minimising in $S_j$. Since each $S_j$ contains a tubular neighbourhood of $\Sigma$ by their construction, we see that the desired statement holds.
\end{proof}

Under the assumption of weak strict stability we also obtain the following statement, which combined with Theorem~\ref{thm: F_H strict stability to VCM} yields the full statement of Theorem~\ref{thm: strict stability to VCM}:

\begin{theorem}\label{thm: weak strictly stable to VCM}
    If $\Sigma$ is weakly strictly stable, then $\Omega$ is uniquely locally volume-constrained minimising in a tubular neighbourhood.
\end{theorem}

\begin{proof}
We emulate the approach in the smooth case carried out in \cite{G96}, showing that weak strict stability of $\Sigma$ implies strict stability for a modified functional (obtained by adding a quadratic term to $\mathcal{F}_H$ that does not affect the first variation). We then show that $\Omega$ uniquely minimises this modified functional, and hence is volume-constrained minimising, in $S_j$ for all sufficiently large $j \geq 1$. Since the approach is similar to that of Theorem~\ref{thm: F_H strict stability to VCM} above, we only outline the necessary modifications.

\bigskip

For $C > 0$ sufficiently large we consider the modified functional
$$\mathcal{E}_C(E) = \mathcal{F}_H(E) + 
\frac{C}{2}\bigl(\mathrm{Vol}_g(E)-\mathrm{Vol}_g(\Omega)\bigr)^2$$
defined on $E \in \mathcal{C}(M)$, and will show the strict stability of $\Sigma$ with respect to $\mathcal{E}_C$ (c.f.~\cite[Lemma 5]{G96}). The compactness argument used in the proof of \cite[Lemma 5]{G96} does not apply directly here, since $W^{1,2,-2}_0(\Sigma)$ does not compactly embed into the space of $\rho^{-2}$-weighted $L^2(\Sigma)$ functions. Instead, write $u=w+(\int_\Sigma u \, dA_g)\,\psi_0$ for a fixed $\psi_0\in C^\infty_c(\mathrm{Reg}(\Sigma))$ with $\int_\Sigma\psi_0 \, dA_g = 1$, so that $\int_\Sigma w \, dA_g = 0$. We then see from the weak strict stability inequality \eqref{eqn: strict stability inequality} applied to $w$, the boundedness of $Q_L$ on $W^{1,2,-2}_0(\Sigma)$, and Young's inequality, that for $C > 0$ sufficiently large we have
$$
\delta^2 \mathcal{E}_C(u) = Q_L(u,u) + C\left(\int_\Sigma u \, dA_g\right)^2  \geq \frac{\lambda}{4}\|u\|_{W^{1,2,-2}(\Sigma)}^2$$
for all $u\in W^{1,2,-2}_0(\Sigma)$. In other words, $\Sigma$ is strictly stable for $\mathcal{E}_C$ with constant $\frac{\lambda}{4}$ when $C > 0$ is sufficiently large.

\bigskip
    
We now show that $\Omega$ is uniquely $\mathcal{E}_C$-minimising, and hence volume-constrained minimising, in the sets $S_j$ constructed above for sufficiently large $j \geq 1$. To this end, for each $j \geq 1$ we let $\Omega_j \in \mathcal{C}(M)$ solve

$$\mathcal{E}_C(\Omega_j) = \inf_{E \in \mathcal{C}(M)} \{ \mathcal{E}_C(E) \, | \, E \Delta \Omega \subset S_j\};$$

we will show that $\Omega_j = \Omega$ for $j \geq 1$ sufficiently large. Just as in the proof of Theorem~\ref{thm: F_H strict stability to VCM} above, by regularity results for the isoperimetric problem with respect to an obstacle (e.g.~see \cite[Theorem 1.1]{V10} for a formulation in Riemannian manifolds) and since $\Omega_j$ is volume-constrained minimising (for its volume) in $S_j$ and $\partial S_j$ is a union of smooth hypersurfaces, we have for some $\alpha \in (0,1)$ that $\Sigma_j = \partial \Omega_j$ is a $C^{1,\alpha}\cap W^{2,2}$ hypersurface away from a closed set of Hausdorff dimension at most $n - 7$ with bounded distributional mean curvature. Exactly as in the proof of Theorem~\ref{thm: F_H strict stability to VCM} above, since $|\Omega_j \Delta \Omega| \rightarrow 0$ as $j \rightarrow \infty$, the lower semicontinuity of the perimeter, \cite[Proposition A.1]{DT13}, and Allard's theorem, we see that $\Sigma_j \rightarrow \Sigma$ as $C^{1,\alpha}$ graphs away from singular points of $\Sigma$ as $j \rightarrow \infty$.

\bigskip

We now apply Lemma~\ref{lem: dichotomy adapted} to the $\Sigma_j$ in conical coordinates for $\Sigma$ near each $p \in \mathrm{Sing}(\Sigma)$. We first note that, since the first variations of $\mathcal{F}_H$ and $\mathcal{E}_C$ differ by a volume difference term, the distributional mean curvature of $\Sigma_j$ satisfies 
$$|H_j - H| = |C (\mathrm{Vol}_g(\Omega_j) - \mathrm{Vol}_g(\Omega))| = d_j$$
wherever it is defined (e.g.~see \cite[Theorem 1.1]{V10}). This means that the mean curvature difference in \eqref{eqn: J monotonicity with mean curvature difference} will not necessarily be zero near $\mathrm{Sing}(\Sigma)$ (unlike in the proof of Theorem~\ref{thm: F_H strict stability to VCM}).
\bigskip

Fix, as in the proof of Theorem~\ref{thm: F_H strict stability to VCM}, $\gamma_p \in (-\frac{n-2}{2},\gamma_1^+(\mathbf{C}_p))$ a distance $\sigma > 0$ from the asymptotic spectrum and $\gamma \in (-\frac{n-2}{2},\min_{p \in \mathrm{Sing}(\Sigma)}\gamma_p)$. For this choice of $\sigma > 0$ and some $\Lambda > \frac{n-2}{2}$ we then apply Lemma~\ref{lem: dichotomy adapted} to obtain $\delta,\widetilde{C} >0$ (writing $C$ for the constant $\widetilde{C}$ there) for which its conclusions hold. Using conical coordinates for $\Sigma$ around each $p \in \mathrm{Sing}(\Sigma)$ (as defined in Subsection~\ref{subsec: notation}) we can write the graphing radii for $\Sigma_j$ over $\Sigma$ around each $p \in \mathrm{Sing}(\Sigma)$ as
    $$
         s_j =  \max_{p \in \mathrm{Sing}(\Sigma)}\inf \{ \tau>0 \, | \, \|v_j\|_{C^1_*(A^g(p;\tau,1))}\le \delta\},
    $$
this ensures for $j \geq 1$ sufficiently large that
    $$\Sigma_j \setminus B_{s_j}^g(p) = \mathrm{graph}_{\Sigma}^g(v_j) \setminus B_{s_j}^g(p).$$ 
Moreover, by the strict stability and strict minimisation assumption on $\mathbf{C}_p$ for each $p \in \mathrm{Sing}(\Sigma)$, we have that both (\ref{eqn: super-quadratic exponent}) holds for $\gamma$ and (\ref{eqn: J monotonicity with mean curvature difference}) holds over each $\mathbf{C}_p$. 

\bigskip

We now aim to show that an analogue of \eqref{eqn: graphing radius estimate} holds (up to a multiplicative constant) where we define $a_j = \Vert v_j\Vert_{L^2(\Sigma \setminus B_{K^{-2}\tilde{r}}^g(\mathrm{Sing}(\Sigma)))}$ for some sufficiently small $0 < \tilde{r} \leq r$ to be chosen. Let us set $b_j = a_j + \widetilde{C}d_j$ for each $j \geq 1$ and note that running the same argument (obtaining a contradiction by applying \eqref{eqn: J monotonicity with mean curvature difference}) as in the proof of Theorem~\ref{thm: F_H strict stability to VCM}, we see that $s_j \leq b_j^{\frac{1}{1-\gamma}}$. It thus remains to show that $d_j \leq C_1 a_j$ since then we have  $s_j^{1-\gamma} \leq b_j = a_j + \widetilde{C}d_j \leq (1 + \widetilde{C}C_1)a_j$,
yielding the desired analogue of \eqref{eqn: graphing radius estimate}.

\bigskip

On $M \setminus B_{\tilde{r}}^g(\mathrm{Sing}(\Sigma))$ we have that $\Sigma_j = \mathrm{graph}_\Sigma^g(v_j)$ and hence we see that 

\begin{equation}\label{eq:bound1}
|(\Omega_j \Delta \Omega) \setminus  B_{\tilde{r}}^g(\mathrm{Sing}(\Sigma))| \leq 2\int_{\Sigma \setminus B_{\tilde{r}}^g(\mathrm{Sing}(\Sigma))} |v_j| \, dA_g \leq C_0 a_j.    
\end{equation}

On each annular region $A^g(p;K^{-1}t,t)$ for $p \in \mathrm{Sing}(\Sigma)$, for $t = K^{-m}\tilde{r}$ for integers $m = 0, \dots, M_t$ where $M_t$ is the largest integer with $K^{-M_t}\tilde{r} \geq K^3s_j$, we see that by Cauchy--Schwarz (and the fact that $n + 2\gamma_p  >0$ by assumption) we have

$$\int_{A^g(p;K^{-1}t,t)} |v_j| \, dA_g \leq \left(\int_{A^g(p;K^{-1}t,t)} \rho^{n + 2\gamma_p} \, dA_g\right)^{\frac{1}{2}} J^{\gamma_p}_{K;\Sigma,g}(v_j;t) \leq Ct^{n + \gamma_p}J^{\gamma_p}_{K;\Sigma,g}(v_j;t).$$
Using \eqref{eqn: J monotonicity with mean curvature difference} repeatedly and summing over annuli, for fixed $j \geq 1$, we see that 
\begin{equation}\label{eq:bound2}
\begin{split}
    |(\Omega_j \Delta \Omega) \cap A^g(p;K^3s_j,\tilde{r})| &\leq 
    C\sum_{m=0}^{M_j}
\int_{\Sigma\cap A^g(p;t_m/K,t_m)}
|v_j|\,dA_g  \leq C\sum_{m=0}^{M_j}
t_m^{n+\gamma_p}J_p(v_j;t_m) \\
& \leq C \max\{ J_p(v_j;\tilde r), \tilde C d_j\} \sum_{m=0}^{M_j}
t_m^{n+\gamma_p}
\leq C_{\tilde{r}}a_j + c\widetilde{C}\tilde{r}^{n + \gamma_p} d_j.
\end{split}
\end{equation}

In $B_{2K^3s_j}^g(\mathrm{Sing}(\Sigma))$ we have that
\begin{equation}
\begin{split}\label{eq:bound3}
    |(\Omega_j \Delta \Omega) \cap  B_{2K^3s_j}^g(\mathrm{Sing}(\Sigma))| &\leq C_2s_j^{n + 1} \leq C_2s_j^{n + \gamma}s_j^{1 - \gamma} \leq C_2 s_j^{n + \gamma}b_j \\
    &= C_2 s_j^{n + \gamma}(a_j + \widetilde{C}d_j) < C_2a_j + C_2\widetilde{C}s_j^{n + \gamma}d_j. 
\end{split}
\end{equation}

With the bounds \eqref{eq:bound1}, \eqref{eq:bound2}, and \eqref{eq:bound3} above, and choosing an initial $0 < \tilde{r} \leq r$ sufficiently small based on the constants, we can absorb the terms involving $d_j$ whenever $j \geq 1$ is sufficiently large (since $s_j \rightarrow 0$ and $n + \gamma > 0$) and obtain the bound $d_j \leq C_1 a_j$ as desired.

\bigskip

With the above estimate in hand, we modify the graphs exactly as in \eqref{eqn: cutoff graphs}, obtaining sets $\widetilde{\Omega}_j$ for each $j \geq 1$ with $\partial \widetilde{\Omega}_j = \widetilde{\Sigma}_j = \mathrm{graph}_\Sigma^g(\widetilde{v}_j)$ and $\widetilde{\Omega}_j \Delta \Omega_j \subset S_j$. We then proceed using Remark~\ref{rem: Taylor expansion} to see that since $\widetilde{\Omega}_j$ differs from $\Omega_j$ only in $B_{2s_j}^g(\mathrm{Sing}(\Sigma))$ we have
$$\mathcal{E}_C(\Omega_j) - \mathcal{E}_C(\Omega) \geq \frac{\lambda}{8}\|\widetilde{v}_j\|_{W^{1,2,-2}(\Sigma)}^2 - \|\widetilde{v}_j\|_{C^1_*(\Sigma)}\|\widetilde{v}_j\|_{W^{1,2,-2}(\Sigma)}^2 - Cs_j^n$$
for some $C > 0$ depending on $\Sigma$ and $(M,g)$. With the estimate $s_j \leq C_1a_j^{\frac{1}{1-\gamma}}$ in hand, since $a_j \leq \|\widetilde{v}_j\|_{L^2(\Sigma)}$ we see that $s_j^n \leq \|\widetilde{v}_j\|_{L^2(\Sigma)}^{2 + \beta}$ for some $\beta > 0$ (using \eqref{eqn: super-quadratic exponent}). Also, choosing $\delta < \frac{\lambda}{32C}$ above and using the fact that $\|\widetilde{v}_j\|_{L^2(\Sigma)} \rightarrow 0$ we see that for $j \geq 1$ sufficiently large we have
$$\mathcal{E}_C(\Omega_j) - \mathcal{E}_C(\Omega) \geq \frac{\lambda}{16}\|\widetilde{v}_j\|_{W^{1,2,-2}(\Sigma)}^2.$$
Just as in the conclusion of the proof of Theorem~\ref{thm: F_H strict stability to VCM}, we thus see that $\Omega_j = \Omega$ for sufficiently large $j \geq 1$ and hence $\Omega$ is uniquely $\mathcal{E}_C$-minimising in $S_j$. In particular, $\Omega$ is volume-constrained minimising in a tubular neighbourhood of $\Sigma$ as desired.
\end{proof}

\begin{remark}
    In \cite{MR10} it was shown that, in ambient dimensions at most seven, weakly strictly stable smooth CMC hypersurfaces are uniquely homologically area minimising for fixed volumes in $L^1$-neighbourhoods. The dimensional restriction stems from their reliance on \cite[Proposition 1]{MR10}. Though we do not pursue it here, it would be interesting to know if our Theorem~\ref{thm: weak strictly stable to VCM} could be utilised to extend their result to allow for isolated singularities of the type considered here.
\end{remark}

\section{Global behaviour}\label{sec: global behaviour}

In this section we construct CMC hypersurfaces with prescribed singular sets by exploiting Theorem~\ref{thm: CMC CHS} in a similar manner to \cite[Theorem 3]{MN26} and \cite[Theorems 1 \& 2]{S00}; in particular, we remove the assumption that the closed manifolds comprising the singular sets there need be oriented. We then use Theorem~\ref{thm: strict stability to VCM} and an adaptation of the arguments in \cite[Corollary 2]{MN26} to construct, under the assumption that all tangent cones are regular, strictly stable and strictly minimising, examples of singular isoperimetric boundaries with non-zeero CMC, thus establishing Theorem~\ref{thm: CMC singular isop}.

\subsection{Constructions over singular minimal hypersurfaces}\label{subsec: constructions}

In order to find strictly stable CMC hypersurfaces with isolated singularities in any closed manifold of dimension at least eight, we use Theorem~\ref{thm: CMC CHS} to modify the constructions of \cite[Theorem 3]{MN26} and \cite[Theorem 1]{S00}:

\begin{proposition}\label{prop: construction in prescribed ambient}
    Suppose that $M$ is a closed manifold of dimension $n + 1 \geq 8$, $L \geq 1$ is an integer, and $\{\mathbf{C}_i\}_{i = 1}^L$ is a collection of strictly stable, non-planar regular minimal hypercones $\mathbf{C}_i \subset \mathbb{R}^{n+1}$. Then, for each $H \neq 0$ sufficiently small, there is a Riemannian metric $g$ on $M$ and an open set $\Omega \subset M$ with $\Sigma = \partial \Omega$ a strictly stable $H$-CMC hypersurface with $L$ isolated singularities. Moreover, the unique tangent cone at each isolated singularity $p_i \in \mathrm{Sing}(\Sigma)$  for $1 \leq i \leq L$ is $\mathbf{C}_i$.
\end{proposition}

\begin{proof}
    Using the construction in the proofs of \cite[Theorem 3]{MN26} (taking each $\Gamma_i$ there as a point) and \cite[Theorem 1]{S00} (noting that one does not require orientability for $M$), and utilising \cite[Proposition 2]{S00}, we find an open set $\Omega_0 \subset M$ and a separating hypersurface $\Sigma_0 = \partial \Omega_0$ which is minimal and strictly stable with respect to some Riemannian metric $g_0$ on $M$. Moreover, $\mathrm{Sing}(\Sigma_0)$ consists of $L$ isolated points with unique tangent cones given by the $\{\mathbf{C}_i\}_{i = 1}^L$, and $g_0$ is Euclidean in $B_1^{g_0}(\mathrm{Sing}(\Sigma_0))$ with $\Sigma_0 \cap B_1(p_i) = \mathbf{C}_i \cap B_2(p_i)$ for each $1 \leq i \leq L$. We note that we do not require the strict minimisation of the $\{\mathbf{C}_i\}_{i = 1}^L$ here since we make no claim about the local area-minimisation of $\Sigma_0$.
    
    \bigskip
    
    By the strict stability of $\Sigma_0$, for all $H \neq 0$ sufficiently small we can invert the Jacobi operator of $\Sigma_0$ away from $B_{1/2}^{g_0}(\mathrm{Sing}(\Sigma_0))$ to find $H$-CMC hypersurfaces with boundary (with zero Dirichlet boundary data on $\Sigma_0 \cap \partial B_{1/2}^{g_0}(\mathrm{Sing}(\Sigma_0))$) as graphs $v_H$ over $\Sigma_0 \setminus B_{1/2}^{g_0}(\mathrm{Sing}(\Sigma_0))$ (c.f.~\cite[Remark 4]{MN26}). 
    
    \bigskip
    
    Applying Theorem~\ref{thm: CMC CHS} with $f = H$ over each $\mathbf{C}_i$ we produce $u_i \in C^{2,\alpha}(\mathbf{C}_i \cap B_1(p_i) \setminus \{p_i\})$ with $\mathrm{graph}_{\mathbf{C}_i}(u_i)$ an embedded strictly stable $H$-CMC hypersurface in $B_1(p_i)$. For each $H \in \mathbb{R} \setminus \{0\}$ sufficiently small as above we construct a hypersurface which agrees with the graphs of the $u_i$ near $\mathrm{Sing}(\Sigma_0)$ and after conformal metric perturbation will be shown to be a strictly stable $H$-CMC hypersurface. 

    \bigskip

    We choose a smooth increasing cutoff function $\eta: [0,\infty) \rightarrow \mathbb{R}$ with $\eta = 0$ on $[0,5/8]$ and $\eta = 1$ on $[7/8,\infty)$, and on $B_1^{g_0}(\mathrm{Sing}(\Sigma_0))$ we set
    
    $$w_i = (1-\eta)u_i + \eta v_H,$$

    so that the graph of $w_i$ interpolates between the graphs of $u_i$ and $v_H$. We denote the resulting hypersurface by $\Sigma$, which is $H$-CMC away from $\cup_{p_i \in \mathrm{Sing}(\Sigma_0)}A^{g_0}(p_i; 5/8,7/8)$, and the open set it bounds by $\Omega$. This is the construction depicted in Figure~\ref{fig: gluings}.

    \bigskip
    
    By performing a conformal metric perturbation supported near the annuli $A^{g_0}(p_i;5/8,7/8)$ (c.f.~\cite[Proposition 1]{BM25} and \cite[Proposition 2.1]{M24}) we produce a Riemannian metric $g$ on $M$ for which $\Sigma$ is a strictly stable $H$-CMC hypersurface. The follows for sufficiently small non-zero $H$ since the metric changes are localised to annuli around $\mathrm{Sing}(\Sigma_0)$ and can be made arbitrarily small.
\end{proof}

\begin{remark}\label{rem: close gluing metrics}
    It follows from the proof above that $g = g_0$ outside of $B_1^{g_0}(\mathrm{Sing}(\Sigma))$ with both metrics Euclidean in $B_{1/2}^{g_0}(\mathrm{Sing}(\Sigma))$. Moreover, one can ensure that the resulting metric is arbitrarily close to the original metric on $M$ in $C^{k,\alpha}$ norm for each $k \geq 1$ and $\alpha \in (0,1)$.
\end{remark}

If one is willing to forgo prescribing the ambient manifold, we can use a similar construction to the above to prescribe arbitrary unions of closed manifolds (not necessarily orientable) as the singular set of a strictly stable $H$-CMC hypersurface for any $H \in \mathbb{R}$:

\begin{proposition}\label{prop: construction with arbitrary singular set}
    Let $n + 1 \geq 8$, $H \in \mathbb{R}$, and for an integer $L \geq 1$ consider collections of closed Riemannian manifolds $\{(\Gamma_i,h_i)\}_{i = 1}^L$ of dimension $k_i$, $\{(X_i,g_i)\}_{i = 1}^L$ of dimension $n_i + 1$ with $k_i + n_i = n$, and a collection $\{\mathbf{C}_i\}_{i = 1}^L$ of strictly stable, non-planar regular minimal hypercones of dimension $n_i$. Then, there is a closed $(n+1)$-dimensional Riemannian manifold $(M,g)$, formed as the connected sum $M = \#_{i = 1}^L(X_i \times \Gamma_i)$ equipped with some Riemannian metric $g$, and an open set $\Omega \subset M$ such that $\Sigma = \partial \Omega \subset (M,g)$ is a strictly stable $H$-CMC hypersurface with
    $$\mathrm{Sing}(\Sigma) = \bigcup_{i = 1}^L \Lambda_i,$$
    where each $\Lambda_i$ is an embedded copy of $\Gamma_i$. Moreover, for some $\sigma > 0$ sufficiently small, in a neighbourhood of each $\Lambda_i$, $(M,g)$ is locally isometric to $B^{n_i+1}_\sigma(0)\times \Gamma_i$ endowed with the product metric $g_{\mathrm{eucl}}\times h_i$, 
    with $\Sigma$ in this region equal to $(\mathrm{graph}_{\mathbf{C}_i}(u_i) \cap B^{n_i+1}_\sigma(0))\times \Gamma_i$ where the $u_i$ are as in Theorem~\ref{thm: CMC CHS}.
\end{proposition}

\begin{proof}
    The construction in the proof of \cite[Theorem 2]{S00} provides the ambient manifold $M$ along with a Riemannian metric $g_0$ and a separating strictly stable minimal hypersurface $\Sigma_0 = \partial \Omega_0 \subset (M,g_0)$ satisfying each of the desired conclusions save for the mean curvature condition. In particular, near each $\Lambda_i$ we have that $\Sigma_0$ equal to $\mathbf{C}_i \times \Gamma_i$.

    \bigskip
    
    As in the proof of Proposition~\ref{prop: construction in prescribed ambient} we use Theorem~\ref{thm: CMC CHS} to produce $H$-CMC hypersurfaces as the graphs of functions $u_i \in C^{2,\alpha}(\mathbf{C}_i \cap B_1(p_i) \setminus \{p_i\})$. We then glue these graphs to $\Sigma_0$ in annular regions with a smooth cutoff and conformally perturb the metric (once to get $H$-CMC, then to get stability) to produce the strictly stable $H$-CMC hypersurface $\Sigma$ as desired.
\end{proof}

\begin{remark}
    Since Theorem~\ref{thm: CMC CHS} allows one to prescribe non-constant mean curvature functions, in principle an adaptation of the above construction yields the existence of prescribed mean curvature hypersurfaces with arbitrary closed manifolds as their singular set; we did not pursue this.
\end{remark}

\subsection{Isoperimetry of singular CMC hypersurfaces}

The following result implies Theorem~\ref{thm: CMC singular isop}: 

\begin{theorem} Suppose $M$ is a closed manifold of dimension $n + 1 \geq 8$, $L \geq 1$ is an integer, and $\{\mathbf{C}_i\}_{i = 1}^L$ is a collection of strictly stable, strictly minimising, non-planar regular minimal hypercones $\mathbf{C}_i \subset \mathbb{R}^{n + 1}$. Then, for each $H \neq 0$ sufficiently small there is a Riemannian metric $h$ on $M$ and a separating, strictly stable $H$-CMC hypersurface $\Sigma \subset (M,h)$ with $L$ isolated singularities which bounds an isoperimetric region. Moreover, the unique tangent cone at each isolated singularity $p_i \in \mathrm{Sing}(\Sigma)$ for $1 \leq i \leq L$ is $\mathbf{C}_i$.
\end{theorem}

\begin{proof}
The proof of Proposition~\ref{prop: construction in prescribed ambient} produces a separating, strictly stable minimal hypersurface $\Sigma_0 = \partial \Omega_0$ for some Riemannian metric $g_0$. Since the $\{\mathbf{C}_i\}_{i = 1}^L$ are assumed to be strictly minimising we see that, as noted in the conclusion of the proof of \cite[Theorem 3]{MN26} (c.f.~\cite[Pages 2323-2324]{S00}), $\Sigma_0$ is uniquely homologically area-minimising in a tubular neighbourhood. Applying \cite[Corollary 1]{MN26} we find a metric $h_0$ on $M$ for which $\Sigma_0$ bounds an isoperimetric region in $(M,h_0)$ with $h_0 = g_0$ in a tubular neighbourhood $V_0 \subset M$ of $\Sigma_0$ (in which $\Sigma_0$ remains locally uniquely homologically area-minimising). 

\bigskip

We now consider a sequence $H_j \rightarrow 0$ (with $H_j \neq 0$) and argue that for $j \geq 1$ sufficiently large we can find $H_j$-CMC hypersurfaces $\Sigma_j = \partial \Omega_j$ with $\Omega_j$ an isoperimetric region for its volume with respect to some Riemannian metric $h_j$. Along this sequence, Proposition~\ref{prop: construction in prescribed ambient} (now applied in $(M,h_0)$) produces separating, strictly stable $H_j$-CMC hypersurfaces $\Sigma_j = \partial \Omega_j$ for some Riemannian metric $h_j$, where $h_j = h_0$ outside of annular regions around $\mathrm{Sing}(\Sigma_0)$ that are contained in $V_0$. Moreover, from its construction (up to choosing larger $j \geq 1$) we can arrange that $\Sigma_j \subset V_0$, each $\Sigma_j$ is homologous to $\Sigma_0$ in $V_0$, with $\Sigma_j \rightarrow \Sigma_0$ in Hausdorff distance, $h_j \rightarrow h$ in $C^3$, and $\mathrm{Per}_{h_j}(\Omega_j) \rightarrow \mathrm{Per}_{h_0}(\Omega_0)$ as $j \rightarrow \infty$.

\bigskip

For a sequence $t_k \rightarrow 0$ we consider the barriers $S_k^j$ (for fixed $j \geq 1$) constructed over each $\Sigma_j$ at the start of Subsection~\ref{subsec: strict stability to VCM}. Applying Theorem~\ref{thm: F_H strict stability to VCM}, we deduce that each $\Omega_j$ is uniquely $\mathcal{F}_{H_j}$-minimising with respect to the metric $h_j$ in $S_k^j$ for all $k \geq 1$ sufficiently large; let us denote $k_j \geq 1$ for the smallest such $k \geq 1$ for which this holds (noting that $S_{k+1}^j \subset S_k^j$). We now show that $\tilde{k} = \sup_{j \geq 1} k_j$ is finite, demonstrating that the tubular neighbourhoods of the $\Sigma_j$ in which the $\Omega_j$ are uniquely volume-constrained minimising are uniformly bounded below in size. If this supremum was not finite we would find a sequence $k_j \rightarrow \infty$ and $\mathcal{F}_{H_j}$-minimisers $E_j \neq \Omega_j$ in $S_{k_j-1}^j$ for each $j \geq 1$. Since $t_{k_j} \rightarrow 0$ we have $\partial S_{k_j -1}^j \rightarrow \Sigma_0$ and hence $E_j \rightarrow \Omega_0$ as $j \rightarrow \infty$. Arguing as in the proof of Theorem~\ref{thm: F_H strict stability to VCM}, now over each hypersurface $\Sigma_j$, we deduce that $E_j = \Omega_j$ for sufficiently large $j \geq 1$ which is a contradiction. Here we are using the fact that the strict stability constants of the $\Sigma_j$ are uniformly bounded below, that the constants in Lemma~\ref{lem: dichotomy adapted} and Remark~\ref{rem: Taylor expansion} do not depend on $j \geq 1$, and that each $\Sigma_j$ has isolated singularities modelled on the same tangent cones $\{\mathbf{C}_i\}_{i = 1}^L$ as $\Sigma_0$.

\bigskip

From the construction of the $\Sigma_j$ over $\Sigma_0$ we ensure that for fixed $k \geq 1$ the barriers $S_k^j$ have a width bounded below independently of $j \geq 1$ and that $\Sigma_j \rightarrow \Sigma_0$ in the Hausdorff sense as $j \rightarrow \infty$. We thus see that $S_{\tilde{k}}^j$ contains a tubular neighbourhood $V$ (with smooth boundary) of $\Sigma_0$ for all $j \geq 1$ sufficiently large with $\Sigma_j \subset V \subset S_{\tilde{k}}^j$ for all $j \geq 1$ sufficiently large. Moreover, from the above we ensure that $\Omega_j$ is uniquely $\mathcal{F}_{H_j}$-minimising in $V$ whenever $j \geq 1$ is sufficiently large.

\bigskip

If $h_j = h_0$ we could directly apply \cite[Corollary 2]{MN26} in $V$ (noting that by \cite[Remark 2]{MN26} we allow for non-empty singular set) to conclude that $\Omega_j$ is an isoperimetric region with respect to  the metric $h_j$ whenever $j \geq 1$ was sufficiently large. However, since we do not necessarily have $h_j = h_0$, we must adapt its proof; we now outline this and refer the reader to the proof of \cite[Corollary 2]{MN26}.

\bigskip

First, we note that the construction of $h_0$ (which agrees with $h_j$ outside of $V$) from the proof of \cite[Theorem 1]{MN26} ensures that any isoperimetric region $E_j$ of enclosed volume $\mathrm{Vol}_{h_j}(\Omega_j)$ with respect to $h_j$ must have some boundary component homologous to $\Sigma_0$ (and hence $\Sigma_j$) in $V_0$. Since $\mathrm{Per}_{h_j}(\Omega_j) \rightarrow \mathrm{Per}_{h_0}(\Omega_0)$ as $j \rightarrow \infty$ and $\Sigma_0$ is uniquely homologically area-minimising in $V$ with respect to $h_0$, potentially taking $j \geq 1$ larger we ensure that $E_j$ has only one boundary component which moreover must lie in $V$ since $\partial E_j \rightarrow \Sigma_0$ as $j \rightarrow \infty$). Indeed, this holds whenever 
$|\mathrm{Per}_{h_j}(\Omega_j) - \mathrm{Per}_{h_0}(\Omega_0)| < \frac{\delta}{2}$,
where $\delta > 0$ is a uniform lower bound for the perimeter of each boundary component of this isoperimetric region. Such a $\delta$ arises from the uniformly bounded geometry of the $h_j$ (noting again that $h_j \rightarrow h_0$ as $j \rightarrow \infty$), which gives a uniform mean curvature bound for $\partial E_j$, from which an application of the monotonicity formula as in the proof of \cite[Theorem 1]{MN26}. 

\bigskip

Thus we see that for $j \geq 1$ sufficiently large, any isoperimetric region $E_j$ of enclosed volume $\mathrm{Vol}_{h_j}(\Omega_j)$ with respect to $h_j$ has only one boundary component, with this component lying in $V$, and homologous to $\Sigma_j$. The unique volume-constrained minimisation of $\Omega_j$ in $V$ then ensures that $E_j = \Omega_j$, so that $\Omega_j$ is an isoperimetric region for its volume with a singular non-zero CMC hypersurface $\Sigma_j$ as its boundary.
\end{proof}

\section*{Acknowledgements} KMS thanks Fritz Hiesmayr for influential discussions regarding \cite{CHS84} and \cite{I17}. GP would like to thank his advisor, Luca Spolaor, for his guidance on isoperimetric problems with singularities.

\section*{Declaration} The ideas presented in this work were developed by the authors.

\bibliographystyle{alpha}
\bibliography{main}

\hrule
\Addresses
\end{document}